\documentclass[12pt]{amsart}
\usepackage{amsmath,xcolor}
\allowdisplaybreaks[4]
\usepackage{geometry,amsthm,graphics,tabularx,amssymb,shapepar}
\usepackage{amscd}
\usepackage{bm,bbm}
\usepackage{mathrsfs}
\usepackage{enumerate}
\usepackage[all]{xypic}

\newcommand{\Res}{\mathrm{Res}}

\newcommand{\fsl}{\mathfrak{sl}}
\newcommand{\fsp}{\mathfrak{sp}}

\newcommand{\lf}{\left}
\newcommand{\rg}{\right}
\newcommand{\g}{\mathfrak g}
\newcommand{\h}{\mathfrak h}

\newcommand{\wh}{\widehat}

\newcommand{\ot}{\otimes}
\newcommand{\op}{\oplus}

\newcommand{\CH}{\mathcal{H}}

\newcommand{\CC}{\mathcal{C}}

\newcommand{\C}{\mathbb{C}}
\newcommand{\R}{\mathbb R}
\newcommand{\N}{\mathbb N}
\newcommand{\Z}{\mathbb Z}

\newcommand{\mk}{\mathbf k}
\newcommand{\Hom}{\mathrm{Hom}}
\newcommand{\nno}{\nonumber}

\newcommand{\te}{\text}
\newcommand{\U}{\mathcal{U}}

\theoremstyle{Theorem}

\theoremstyle{Theorem}

\newtheorem{thm}{Theorem}[section]
\newtheorem{prpt}[thm]{Proposition}
\newtheorem{lem}[thm]{Lemma}

\newtheorem{definition}[thm]{Definition}

\allowdisplaybreaks

\numberwithin{equation}{section}
\title[vertex algebras and  TKK algebras]{Vertex algebras and TKK algebras}

\author{Fulin Chen$^1$}
\address{School of Mathematical Sciences, Xiamen University,
 Xiamen, China 361005} \email{chenf@xmu.edu.cn}
 \thanks{$^1$Partially supported by China NSF grant (Nos. 11971397, 12131018).}
 \author{Lingen Ding}\address{School of Mathematical Sciences, Xiamen University,
 Xiamen, China 361005} \email{lingen@stu.xmu.edu.cn}
  \author{Qing Wang$^2$}\address{School of Mathematical Sciences, Xiamen University,
 Xiamen, China 361005} \email{qingwang@xmu.edu.cn }\thanks{$^2$ Partially supported by
 China NSF grants (Nos. 12071385, 12161141001).}

\subjclass[2010]{17B67 \& 17B69}
\keywords{TKK algebra, vertex algebra, twisted module}

\begin{document}

\begin{abstract}
In this paper, we associate the TKK algebra $\widehat{\mathcal{G}}(\mathcal{J})$ with vertex algebras through twisted modules.
Firstly, we prove that for any complex number $\ell$, the category of restricted $\widehat{\mathcal{G}}(\mathcal{J})$-modules of level $\ell$ is canonically isomorphic to the category of $\sigma$-twisted $V_{\widehat{\CC_\g}}(\ell,0)$-modules, where $V_{\widehat{\CC_\g}}(\ell,0)$ is a vertex algebra arising from the
toroidal Lie algebra of type $C_2$ and $\sigma$ is an isomorphism of $V_{\widehat{\CC_\g}}(\ell,0)$ induced from the involution of this toroidal Lie algebra.
    Secondly, we prove that for any nonnegative integer $\ell$, the integrable restricted $\widehat{\mathcal{G}}(\mathcal{J})$-modules of level $\ell$ are exactly the
    $\sigma$-twisted modules for the quotient vertex algebra $L_{\widehat{\CC_\g}}(\ell,0)$ of $V_{\widehat{\CC_\g}}(\ell,0)$. Finally, we
    classify the irreducible $\frac{1}{2}\N$-graded $\sigma$-twisted $L_{\widehat{\CC_\g}}(\ell,0)$-modules.
\end{abstract}
\maketitle

\section{Introduction}
Extended affine Lie algebras are a class of Lie algebras that  generalize finite-dimensional simple Lie algebras and
affine Kac-Moody algebras (see \cite{H-KT,AABGP,N} and the references therein).
To characterize the (extended affine) root system of extended affine Lie algebras, a notion of semi-lattice was introduced in \cite{AABGP}.
By definition, a semi-lattice $S$ is a subset of a lattice $\Lambda$ in $\R^\nu$ $(\nu\ge 1)$ such that $0\in S$, $S$ generates $\Lambda$, and $2a-b\in S$ for all $a,b\in S$.
Starting from a semi-lattice $S$,
one can associate a Jordan algebra $\mathcal{J}(S)$ to it  (see \cite{AABGP}), and then obtain a perfect Lie algebra $\mathcal{G}(\mathcal{J}(S))$
from  the  Tits-Kantor-Koecher construction (see \cite[Chapter 5]{J}).
The universal central extension $\widehat{\mathcal{G}}(\mathcal{J}(S))$ of  $\mathcal{G}(\mathcal{J}(S))$   is usually referred as the TKK algebra
obtained from the semi-lattice $S$.
In the literatures, the structure and representation theory of the TKK algebras have been extensively studied (see \cite{BL,CCT,CT,GJ,GZ,KT,MT1,MT2,Y,T,YT}, etc.).

When
$S$ is the  unique (up to isomorphism) non-lattice semi-lattice in $\R^2$,  following \cite{T},
we call  $\widehat{\mathcal{G}}(\mathcal{J}(S))$ the baby TKK algebra  and denote it
by $\widehat{\mathcal{G}}(\mathcal{J})$ for simplicity.
In this case the root system obtained from $\widehat{\mathcal{G}}(\mathcal{J})$ is the ``smallest" extended affine root system that is neither finite nor affine (see \cite{AABGP}).
It is well-known that affine Kac-Moody algebras through their restricted modules
can be naturally associated with vertex algebras and (twisted) modules (see \cite{FZ,FLM,Li1,Li2}).
These associations play an
important role in both affine Kac-Moody algebra theory and vertex (operator) algebra theory.
The main goal of this paper is to
generalize such associations to the baby TKK algebra.

Let $\mathfrak{t}$ be a $2$-toroidal Lie algebra (see \cite{MRY}).
It was shown in  \cite{CJKT} that there is a $\Z$-graded Lie conformal algebra
$\CC_\g$ whose associated Lie algebra $\wh{\CC_\g}$ is isomorphic to $\mathfrak{t}$.
As in the affine case, for any complex number $\ell$,
one can construct an $\N$-graded (universal) vertex algebra $V_{\wh{\CC_\g}}(\ell,0)$ associated to $\CC_\g$.
When $\mathfrak{t}$ is of type $C_2$, we define a graded involution $\sigma$ of $\CC_\g$,
which induces an involution $\widehat\sigma$ of $\widehat{\CC_\g}$ and a
graded  involution $\sigma$
of $V_{\wh{\CC_\g}}(\ell,0)$.
As the first main result of this paper, we prove in Theorem \ref{rem} that there is a natural correspondence between the category of restricted $\widehat{\mathcal{G}}(\mathcal{J})$-modules
of level $\ell$ and that of $\sigma$-twisted $V_{\wh{\CC_\g}}(\ell,0)$-modules.
The proof of this correspondence is based on a new realization of $\widehat{\mathcal{G}}(\mathcal{J})$ as
the fixed point subalgebra of the toroidal Lie algebra $\mathfrak{t}$ under $\widehat\sigma$.

For any nonnegative integer $\ell$, we also introduce a $\sigma$-stable graded ideal $J_{\wh{\CC_\g}}(\ell,0)$ of $V_{\wh{\CC_\g}}(\ell,0)$.
Then $\sigma$ induces a  graded involution, still denote it by $\sigma$, of the $\N$-graded quotient vertex algebra $L_{\wh{\CC_\g}}(\ell,0)=V_{\wh{\CC_\g}}(\ell,0)/J_{\wh{\CC_\g}}(\ell,0)$.
As the second main result of this paper, we prove in Theorem \ref{in-re-m} that there is an isomorphism between the category of integrable restricted $\widehat{\mathcal{G}}(\mathcal{J})$-modules
of level $\ell$ and that of $\sigma$-twisted $L_{\wh{\CC_\g}}(\ell,0)$-modules.
Furthermore, together with the integrable highest
weight $\widehat{\mathcal{G}}(\mathcal{J})$-modules constructed in \cite{CCT} (see also \cite{CT}), we classify  the irreducible $\frac{1}{2}\N$-graded $\sigma$-twisted $L_{\wh{\CC_\g}}(\ell,0)$-modules in terms of integrable highest
weight $\widehat{\mathcal{G}}(\mathcal{J})$-modules (see Theorem \ref{g-irr-in-re-m}).

Associated to $\widehat{\mathcal{G}}(\mathcal{J})$, there is an extended affine Lie algebra $\widetilde{\mathcal{G}}(\mathcal{J})$
such that its core is isomorphic to $\widehat{\mathcal{G}}(\mathcal{J})$.
In \cite[Theorem 6.1]{BL}, the algebra $\widetilde{\mathcal{G}}(\mathcal{J})$ is realized as
the fixed point subalgebra of a toroidal extended affine Lie algebra under two commutating automorphisms $\sigma_0$ and $\sigma_1$.
Based on this realization, a class of irreducible $\widetilde{\mathcal{G}}(\mathcal{J})$-modules were constructed by using certain $\sigma_0$-twisted modules for a quotient of the full toroidal vertex algebra.
By taking restriction, \cite[Theorem 6.1]{BL} yields a realization of $\widehat{\mathcal{G}}(\mathcal{J})$  as
the fixed point subalgebra of $\mathfrak{t}$ under the two commutating automorphisms $\sigma_0|_{\mathfrak{t}}$ and $\sigma_1|_{\mathfrak{t}}$.
However,   we can not obtain the desired associations between $\hat{\mathcal{G}}(\mathcal{J})$ and the toroidal vertex algebra based on this realization.
This why we present in this paper a new realization of $\widehat{\mathcal{G}}(\mathcal{J})$  as
the fixed point subalgebra of $\mathfrak{t}$ under a {\em single} automorphism $\hat\sigma$.

Now we give an outline of this paper. In Section 2, we introduce the definition of the baby TKK algebra $\widehat{\mathcal{G}}(\mathcal{J})$ and
present a triangular decomposition of it.
In Section 3, we review the notion of twisted Lie algebras of Lie conformal algebras  as well as the notion of twisted modules of vertex algebras.
The connection between these two notions are also formulated.
In Section 4, we associate restricted $\widehat{\mathcal{G}}(\mathcal{J})$-modules of level $\ell$ to $\sigma$-twisted modules of $V_{\wh{\CC_\g}}(\ell,0)$.
In Section 5, we associate integrable restricted $\widehat{\mathcal{G}}(\mathcal{J})$-modules of level $\ell$ to $\sigma$-twisted $L_{\wh{\CC_\g}}(\ell,0)$-modules and classify the irreducible $\frac{1}{2}\N$-graded $\sigma$-twisted $L_{\wh{\CC_\g}}(\ell,0)$-modules.

In this paper, we let $\Z$, $\Z_{+}$, $\Z^{\times}$, $\N$ and $\C$ denote
 the set of integers, positive integers, nonzero integers, nonnegative integers and complex numbers, respectively.

\section{Baby TKK algebra}
In this section, we review some basics on the baby Tits-Kantor-Koecher algebra (cf.\,\cite{AABGP,T}).

 First we recall the Jordan algebra associated to the non-lattice semi-lattice in $\R^2$.
Set $e_1=(1,0)$ and $e_2=(0,1)$ in $\Z^2$, and also set
\begin{equation*}
S_0=2\Z^2, \quad S_1=e_1+2\Z^2, \quad S_2=e_2 + 2\Z^2.
\end{equation*}
Then $S=S_0\cup S_1 \cup S_2$ is a semi-lattice in $\Z^2$.
It is known that, associated to $S$, there is a  Jordan algebra  (\cite{AABGP})
\[\mathcal{J}=\op_{\rho \in S}\,\C\, x^{\rho}\qquad(\text{$
 x^{\rho}$ is a symbol})\]
with the multiplication  given by
\begin{eqnarray*}x^{\rho }\cdot x^{\tau}=\left\{\begin{array}{ll}
                   x^{\rho +\tau}, &  \text{ if }\ \rho +\tau\in S,\\
                   0, &   \text{ if }\ \rho +\tau\in S^\bot,
                 \end{array}\right.
\end{eqnarray*}
where $S^{\bot}:=\Z^2 \setminus S=e_1+e_2 + 2\Z^2$.

For each $a\in \mathcal{J}$, write $L_a$ for the  left multiplication  by $a$.
Denote by
$$
\text{Inder}(\mathcal{J})=\text{Span}\{[L_a,L_b]\mid a,b\in \mathcal{J}\}
$$
the space of inner derivations on $\mathcal{J}$, and denote by
$\mathcal{I}$ the subspace of $\mathcal{J}\ot\mathcal{J}$ spanned by the following elements:
\begin{equation*}
a\ot b+b\ot a \quad \text{and} \quad a\cdot b\ot c +b\cdot c\ot a+c\cdot a\ot b,
\end{equation*}
where $a,b,c\in \mathcal{J}$. Form the quotient space
$$\langle \mathcal{J},\mathcal{J} \rangle :=\mathcal{J}\ot \mathcal{J}/\mathcal{I}.$$
For $a,b\in \mathcal{J}$, we will use the notation  $\langle a,b\rangle$ to denote
 the image of $a\ot b$ under the quotient map $\mathcal{J}\ot\mathcal{J}\longrightarrow \langle \mathcal{J},\mathcal{J}\rangle$.

As usual, we write $\mathfrak{sl}_2$ for the $3$-dimensional complex simple Lie algebra.
  The baby TKK algebra is by definition the Lie algebra (\cite{T})
\[\widehat{\mathcal{G}}(\mathcal{J})=(\mathfrak{sl}_2\ot \mathcal{J})\op \langle \mathcal{J},\mathcal{J} \rangle\]
with the Lie bracket given by
\begin{equation}
 \begin{split}
 &[A\ot a,B\ot b]= [A,B]\ot ab+2\text{tr}(AB)\langle a,b \rangle,
\\
& [\langle a,b \rangle,A\ot c]= A\ot [L_a,L_b]c ,\\
& [\langle a,b \rangle,\langle c,d \rangle]= \langle [L_a,L_b]c,d \rangle+\langle c,[L_a,L_b]d \rangle,\\
 \end{split}\end{equation}
where $a,b,c,d \in \mathcal{J}$, $A,B\in\mathfrak{sl}_2$ and $\mathrm{tr}$ is the trace form.
It is known that $\widehat{\mathcal{G}}(\mathcal{J})$ is the universal central extension of the TKK construction associated to $\mathcal{J}$ (\cite{T}).

We fix a standard Chevalley basis
$\{x_{+},x_{-},\alpha^{\vee} \}$ of $\mathfrak{sl}_2$ such that
\begin{equation*}
[x_{+},x_{-}]=\alpha^{\vee},\quad  [\alpha^{\vee},{x_{+}}]=2x_{+}\quad\text{and}\quad
[\alpha^{\vee},{x_{-}}]=- 2x_{-}.
\end{equation*}
For $\rho=(m,n)\in \Z^2$ and $i=1,2$, set
 \begin{equation*}\begin{split}
&x_{\pm}(\rho )=x_{\pm}(m,n):= \left\{\begin{array}{ll}
x_{\pm}\ot x^{\rho }, &\qquad\text{ if } \rho \in S,\\
                          0,&\qquad\text{ if } \rho \in S^{\bot},\\
                         \end{array}
                         \right.\\
&a^{\vee}(\rho )=\alpha^{\vee}(m,n):=\left\{\begin{array}{ll}
\alpha^{\vee} \ot x^{\rho }, &\text{ if } \rho \in S,\\
                         2\langle \ x^{e_1},x^{\rho -e_1} \rangle, & \text{ if } \rho \in S^{\bot},\\
                         \end{array}
                         \right.\\
&C_i(\rho)=C_i(m,n) :=\left\{\begin{array}{ll}
\langle \ x^{e_i},x^{\rho -e_i} \rangle, &\quad\text{ if } \rho \in S_0,\\
                         0, &\quad\text{ if } \rho  \notin S_{0}, \\
                         \end{array}
                         \right.\\
   &\Omega(\rho )=\Omega(m,n):=\frac{1}{2}((-1)^{m}-(-1)^{n})=\left\{\begin{array}{ll}
   -1,  & \text{ if } \rho  \in S_{1}, \\
    1,  & \text{ if } \rho  \in S_{2},\\
    0, &\text{ if } \rho \in S_0\cup S^{\bot}.\\
                         \end{array}
                         \right.
\end{split}\end{equation*}
Then the Lie algebra $\widehat{\mathcal{G}}(\mathcal{J})$ is spanned by the set
$$\{x_{\pm}(\rho ),\,\alpha^{\vee}(\tau),\,C_i(\upsilon)\mid\rho\in S,\,\tau\in \Z^2, \,\upsilon\in S_0,\ i=1,2\}$$
and subject to the following relations:
\begin{equation*}\begin{split}
(\text{R1}) \enspace \text{For}\ \rho&,\,\tau  \in\Z^2,\\
[\alpha^{\vee}&(\rho),\alpha^{\vee}(\tau)]
=\left\{\begin{array}{lll}
        -4\sum\limits_{i=1,2}\,(\rho\cdot e_i)\,C_i(\rho+\tau), &\text{ if } \rho ,\,\tau\notin S,\\
       4\sum\limits_{i=1,2}\,(\rho\cdot e_i)\,C_i(\rho+\tau), &\text{ if } \rho ,\,\tau\in S , \ \rho+\tau\in S,\\
        2\Omega(\tau)\, \alpha^{\vee}(\rho+\tau), &\text{ if } \rho\notin S,\,\tau\in S \ \text{or}\  \rho,\,\tau\in S , \ \rho+\tau\notin S .\\
          \end{array}
\right.\\
(\text{R2}) \enspace \text{For}\ \rho&\in \Z^2,\, \tau  \in S,\\
&[\alpha^{\vee}(\rho),x_{\pm}(\tau)]=\left\{\begin{array}{lll}
       \pm 2x_{\pm}\,(\rho+\tau), &\text{ if } \rho \in S, \\
       2\Omega(\tau)\,x_{\pm}\,(\rho+\tau), &\text{ if } \rho \notin S. \\
          \end{array}
\right.\\
(\text{R3}) \enspace \text{For}\ \rho&,\, \tau  \in S,\\
&[x_{\pm}(\rho),x_{\pm}(\tau)]=0,\\
&[x_{+}(\rho),x_{-}(\tau)]=\left\{\begin{array}{lll}
       \Omega(\tau)\alpha^{\vee}(\rho+\tau), &\text{ if } \rho +\tau\notin S ,\\
       \alpha^{\vee}(\rho+\tau) + 2\sum\limits_{i=1,2}(\rho\cdot e_i)C_{i}(\rho+\tau), &\text{ if } \rho+\tau \in S.\\
     \end{array}
\right.\\
(\text{R4}) \enspace \text{For}\ \upsilon&\in S_0,\, i=1,2,  \  C_i(\upsilon) \ \text{are central and subject to } \\
&\hspace{3cm} (\upsilon\cdot e_1)C_1(\upsilon)+(\upsilon\cdot e_2)C_2(\upsilon)=0.
\end{split}
\end{equation*}
where, for $a=(a_1,a_2)$ and $b=(b_1,b_2)$ in  $\Z^2$,  $a\cdot b=a_1b_1+a_2b_2$ is their standard inner product.

By adding two derivations $\text{d}_1,\text{d}_2$ to  $\widehat{\mathcal{G}}(\mathcal{J})$, one obtains the extended baby TKK algebra
$$\mathscr{L}=\widehat{\mathcal{G}}(\mathcal{J})\op \C\text{d}_1\op \C\text{d}_2,$$
where
$$
[\text{d}_i,A\ot x^{\rho}]=(\rho\cdot e_i)\,A\ot x^{\rho},\quad
[\text{d}_i,\langle x^{\rho},\,x^{\tau}\rangle]=((\rho+\tau)\cdot e_i)\,\langle x^{\rho},x^{\tau}\rangle$$
for $A\in \mathfrak{sl}_2,\, \rho,\,\tau\in \Z^2$  and  $i=1,2$.
We fix a $\Z$-grading on $\widehat{\mathcal{G}}(\mathcal{J})$ with respect to the adjoint action of $-\text{d}_1$. Namely,
\begin{equation}\label{zgradtkk}
\widehat{\mathcal{G}}(\mathcal{J})=\op_{n\in\Z}\widehat{\mathcal{G}}(\mathcal{J})_{(n)}\,, \ \text{ where}\ \
 \widehat{\mathcal{G}}(\mathcal{J})_{(n)}=\{u\in \widehat{\mathcal{G}}(\mathcal{J})\mid[\text{d}_1,u]=-nu\}.
\end{equation}

Set
\begin{equation*}
\h=\C\alpha^{\vee}\op \C C_1\op \C C_2 \op \C \text{d}_1 \op   \C\text{d}_2,
\end{equation*}
 an abelian subalgebra of  $\mathscr{L}$.
 Let $\alpha$ be the simple root of $\mathfrak{sl_2}$ for which $\alpha(\alpha^{\vee})=2$. Extend $\alpha$ to an element of $\h^{*}$ (the dual space of $\h$) such that
 $$
 \alpha(C_{i})=\alpha(\text{d}_i)=0 \quad \text{for}\ i=1,2.
 $$
Define  $\delta_{1},\delta_{2}\in \h^{*}$ by letting
\begin{equation*}
\delta_{i}(\alpha^{\vee})=0,\quad\delta_{i}(C_{j})=0,\quad \delta_{i}(\text{d}_j)=\delta_{i,j}\quad\text{for} \ i,j=1,2.
\end{equation*}

 With respect to $\h$,  $\mathscr{L}$ admits a root space decomposition as follows:
\begin{equation*}
\mathscr{L}=\bigoplus_{\beta\in\Delta}\mathscr{L}_{\beta},\qquad \mathscr{L}_{\beta}:=\{x\in \mathscr{L}\mid [h,x]=\beta(h)x ~~\text {for ~~~ all} ~~h\in \h \}.
\end{equation*}
where \[\Delta=\{\pm \alpha +m_1\delta_1+m_2\delta_2\mid (m_1,m_2)\in S\}\cup \{n_1\delta_1+n_2\delta_2\mid (n_1,n_2)\in \Z^2\}.\]
On the other hand, by setting
 $\widehat{\mathcal{G}}(\mathcal{J})_{\beta}=\widehat{\mathcal{G}}(\mathcal{J})\cap \mathscr{L}_{\beta}\,$ for all $\beta\in \Delta$, we also have that
  \begin{equation*}
\widehat{\mathcal{G}}(\mathcal{J})=\bigoplus_{\beta\in\Delta}\widehat{\mathcal{G}}(\mathcal{J})_{\beta}.
\end{equation*}

We fix a decomposition
\begin{equation*}
\Delta=\Delta_{+}\cup\Delta_{0}\cup \Delta_{-}
\end{equation*}
 of $\Delta$, where
\begin{eqnarray*}
\Delta_+&=&\{\alpha+n\delta_2,\pm \alpha+m\delta_1+n\delta_2,m\delta_1+n\delta_2\mid m,n\in \Z\ \text{with}\ m>0\}\cap \Delta,\\
\Delta_-&=&\{-\alpha+n\delta_2,\pm \alpha+m\delta_1+n\delta_2,m\delta_1+n\delta_2\mid m,n\in \Z\ \text{with}\ m<0\}\cap \Delta,
\end{eqnarray*}
and $\Delta_{0}=\{n\delta_2\mid n\in \Z\}$.
It induces  a triangular decomposition
$$
\widehat{\mathcal{G}}(\mathcal{J})=
\widehat{\mathcal{G}}(\mathcal{J})^{+}
\op \mathcal{H} \op
\widehat{\mathcal{G}}(\mathcal{J})^{-}
$$
 of $\widehat{\mathcal{G}}(\mathcal{J})$, where
$\widehat{\mathcal{G}}(\mathcal{J})^{\pm}
=\bigoplus_{\beta\in\Delta_{\pm}}\widehat{\mathcal{G}}(\mathcal{J})_{\beta}$ and $\mathcal{H}=\bigoplus_{\beta\in\Delta_{0}}\widehat{\mathcal{G}}(\mathcal{J})_{\beta}$.
Explicitly, we have that
\begin{equation*}
\begin{split}
\widehat{\mathcal{G}}(\mathcal{J})^{\pm}&=
\sum_{\substack{(m,n)\in S\\ \pm m\geq 0,\ n\in\Z}}\C x_{\pm}(m,n)\  \op\sum_{\substack{(m,n)\in S\\ \pm m> 0,\ n\in\Z}}\C x_{\mp}(m,n)
\ \op  \sum_{ \pm m> 0,\ n\in\Z}\C \alpha^{\vee}(m,n)\\
 &\op\sum_{ \pm m> 0,\ n\in\Z^{\times}}\C C_1(2m,2n)\op\sum_{ \pm m> 0}\C C_2(2m,0),
\end{split}
\end{equation*}
 and $$\CH=\sum_{n\in \Z}\C \alpha^{\vee}(0,n)\op\sum_{n\in \Z}\C C_1(0,2n)\op \C C_2.$$

\section{Lie conformal algebras and twisted modules for vertex algebras}
In this section,  we recall and present some of results on Lie conformal algebras and twisted modules for vertex algebras.
We start with the following definition of Lie conformal algebra \cite{K}, also known as a vertex Lie algebra \cite{P} (cf. \cite{DLM1}).

\begin{definition}
{\em A {\em Lie conformal algebra} is a vector space $\mathcal{C}$, together with an endomorphism $\partial$ of $\mathcal{C}$, and  a linear map
\begin{eqnarray}
\label{E1}
Y_{-}(\cdot,z): \mathcal{C}\rightarrow\Hom(\mathcal{C},z^{-1}\mathcal{C}[z^{-1}]),\quad a\mapsto Y_{-}(a,z)=\sum_{n\ge0}a_{n}z^{-n-1}
\end{eqnarray}
such that
\begin{eqnarray}
\label{E2}&&[\partial,Y_{-}(a,z)]=\frac{d}{dz}Y_{-}(a,z)=Y_{-}(\partial a,z),\\
\label{E3}&&Y_{-}(a,z)b=\te{Sing}(e^{z\partial }Y_{-}(b,-z)a),\\
\label{E4}&&[Y_{-}(a,z_{1}),Y_{-}(b,z_{2})]=\te{Sing}(Y_{-}(Y_{-}(a,z_{1}-z_{2})b,z_{2})),
\end{eqnarray}
where $a,b\in \mathcal{C}$ and \te{Sing} stands for the singular part.  }
\end{definition}

The following result can be found in \cite[Remark 4.2]{P}.

\begin{lem}\label{cla-la}
Let $\mathcal{C}$ be a vector space equipped with a linear operator $\partial:\ \mathcal{C}\rightarrow\mathcal{C}$ and a linear map $Y_{-}$ as in \eqref{E1} such that \eqref{E2} holds.  Then $\mathcal{C}$ is a Lie conformal  algebra if and only if
\begin{eqnarray}
\widehat{\mathcal{C}}=(\mathcal{C}\otimes \C[t,t^{-1}])/\te{Image}(\partial\otimes 1+1\otimes \frac{d}{dt})
\end{eqnarray} is a Lie algebra under the
 Lie bracket
\begin{eqnarray}
\label{lb}[a(m),b(n)]=\sum_{i\ge0}\binom{m}{i}(a_{i}b)(m+n-i)
\end{eqnarray}
for $a,b\in \mathcal{C},\ m,n\in\Z,$ where
$\binom{m}{i}=\frac{m(m-1)\cdots(m-i+1)}{i!}$ and
$a(m)$ denotes the image of $a\ot t^{m}$ under the  quotient map $\CC \ot \C[t,t^{-1}]\rightarrow \widehat{\CC}.$
\end{lem}

Note that in terms of generating functions
$$
a(z)=\sum_{n\in\Z}a(n)z^{-n-1}\quad \text{for}\  a\in \CC,
$$ the relation \eqref{lb} can be written as
\begin{eqnarray}
[\,a(z_{1}), b(z_{2})\,]=\sum_{i\ge 0}(a_{i}b)(z_{2})\,\frac{1}{i!}\left(\frac{\partial}{\partial z_{2}}\right)^{i}\left(z_{1}^{-1}\delta\left(\frac{z_{2}}{z_{1}}\right)\right).
\end{eqnarray}

Let $\CC$ be a Lie conformal algebra.  Set
$$\widehat{\CC}_{-}=\text{Span}\{a(n)\mid a\in \CC, n<0\}\quad \text{and}\quad \widehat{\CC}_{+}=\text{Span}\{a(n)\mid a\in \CC, n\geq 0\}.$$
It is obvious that
$
 \widehat{\CC}= \widehat{\CC}_{-}\op \widehat{\CC}_{+}
$ is
a direct sum of Lie subalgebras.
And, it was proved in \cite[Theorem 4.6]{P} that the map
\begin{eqnarray}
\label{id-c}\mathcal{C}\rightarrow \widehat{\mathcal{C}}_{-},\quad a\mapsto a(-1)
\end{eqnarray}
is a linear isomorphism. Form the induced $\widehat{\mathcal{C}}$-module
$$V_{\widehat{\CC}}=\mathcal{U}(\widehat{\mathcal{C}})\otimes_{\mathcal{U}(\widehat{\mathcal{C}}_{+})}\C,$$
where  $\C$ is viewed as the trivial  $\widehat{\mathcal{C}}_{+}$-module.
Set $\mathbbm{1}=1\otimes 1$. Identify $\CC$ as a subspace of $V_{\widehat{\mathcal{C}}}$ by the linear map $a\mapsto a(-1)\mathbbm{1}$.
 From \cite{P}, there exists a unique vertex algebra structure on $V_{\widehat{\mathcal{C}}}$ with $\mathbbm{1}$ as the vacuum vector and with
$Y(a,z)=a(z)$
for $a\in\CC $.

Note that the operator $\partial$ on $\CC$ induces a derivation, still called $\partial$, of $\widehat{\mathcal{C}}$ such that
\[\partial(a(n))=(\partial a)(n)=-na(n-1)\quad \te{for }a\in\CC,\ n\in\Z.\]
Extend $\partial$  to an (associative) derivation of $\U(\widehat{\mathcal{C}})$.
Since $\partial(\widehat{\mathcal{C}}_{+})\subset \widehat{\mathcal{C}}_{+}$, it induces a linear operator
on $V_{\widehat{\mathcal{C}}}$, which coincides with the canonical derivation $\mathcal{D}$ on $V_{\widehat{\mathcal{C}}}$ defined by $v\mapsto v_{-2}\mathbbm{1}$.

 A {\em homomorphism} of a  Lie conformal algebra $\CC$ is a linear endomorphism  $\sigma$ on $\CC$ such that
 \begin{eqnarray*}
\sigma\circ \partial=\partial\circ\sigma\quad\te{and}
\quad\sigma (Y_{-}(a,z)b)=Y_{-}(\sigma (a),z)\sigma(b)\quad\te{for } a,b\in\mathcal{C}.
\end{eqnarray*}
The following result was proved in \cite[Lemma 3.10]{CLiTW}.
\begin{lem}\label{au-cla} Let $\mathcal{C}$ be a Lie conformal algebra and let $\sigma$
be a linear automorphism of $\mathcal C$ viewed as a vector space
such that $\sigma\circ \partial=\partial\circ\sigma$. Then $\sigma$ is an automorphism of $\mathcal{C}$ if and only if
the map
\begin{eqnarray*}
\widehat{\sigma}:\widehat{\mathcal{C}}\rightarrow\widehat{\mathcal{C}},\quad a(m)\mapsto\sigma(a)(m)\quad  (a\in\mathcal{C},m\in\Z)
\end{eqnarray*}
is a Lie automorphism of
 $\widehat{\mathcal{C}}$.
\end{lem}



Suppose  that $\CC$ is a Lie conformal algebra and $\sigma$ is an automorphism of $\CC$  with finite order $T$.
Following \cite{KW}, there is a  Lie algebra
\begin{eqnarray*}
\widehat{\mathcal{C}}[\sigma]=\bigoplus_{r\in\Z}(\CC^{(r)}\otimes t^{\frac{r}{T}})/\te{Image}(\partial\otimes1+1\otimes \frac{d}{dt})
\end{eqnarray*}
associated to the pair $(\mathcal{C},\sigma)$,
where $\CC^{(r)}=\{a\in \CC\mid\sigma(a)=\omega^{r}a\}$ and $\omega=e^{2\pi\sqrt{-1}/T}$.
When $\sigma=\mathrm{id}$, it coincides with the  algebra $\wh\CC$ defined in Lemma \ref{cla-la}.
For  $r\in \Z$ and $a \in\CC^{(r)}$, set
 $$a\left(\frac{r}{T}\right):=a\ot t^{\frac{r}{T}}+\te{Image}(\partial\otimes1+1\otimes \frac{d}{dt}).$$ Then
 the Lie bracket on $\widehat{\mathcal{C}}[\sigma]$ is determined by
\begin{eqnarray}\label{t-lb}
[\,a\left(\frac{r}{T}\right),\,b\left(\frac{s}{T}\right)\,]=\sum_{i\ge 0}\binom{\frac{r}{T}}{i}(a_{i}b)\left(\frac{r+s}{T}-i \right),
\end{eqnarray}
where $a\in\mathcal{C}^{(r)},\ b\in\mathcal{C}^{(s)}$ and $r,\,s\in \Z$.

  For every $a\in \CC$, form the generating function
\begin{align}
a^{\sigma}(z)=\sum_{r\in\Z} a^{(r)}\left(\frac{r}{T}\right) z^{-\frac{r}{T}-1},
\end{align}
where $a^{(r)}=\frac{1}{T}\sum\limits_{k=0}^{T-1}\omega^{-rk}\sigma^k(a)\in \CC^{(r)}$.

\begin{definition}{\em A $\widehat{\mathcal{C}}[\sigma]$-module $W$ is said to be {\em restricted } if
 $$a^{\sigma}(z)\in \Hom(W,W((z^{\frac{1}{T}}))) \quad \text{for all}\ a\in\CC . $$}
\end{definition}

In view of Lemma \ref{au-cla}, we know that $\sigma$ can be uniquely extended  to an automorphism, still called $\sigma$, of vertex algebra $V_{\widehat{\CC}}$
such that
\begin{align}\label{extendsigma}
\sigma(\mathbbm{1})=\mathbbm{1}\quad\te{and}\quad\sigma(a^{1}_{n_{1}}\cdots a^{r}_{n_{r}}\mathbbm{1})=(\sigma a^{1})_{n_{1}}\cdots (\sigma a^{r})_{n_{r}}\mathbbm{1},
\end{align}
where  $a^{i}\in\mathcal{C},\ n_{i}\in\Z$ and $r\in\Z_{+}.$
Note that  $\sigma$ still has order $T$ as an automorphism of $V_{\widehat{\CC}}$.

\begin{definition} {\em Let $(V,Y,\mathbbm{1})$ be a vertex algebra,  and
$\sigma$ be an
automorphism of $V$ with finite order $T$.
  A {\em $\sigma$-twisted $V$-module}
 is a vector space $W$ equipped with a linear map
\begin{eqnarray*}
Y_{W}(\cdot,z):V\rightarrow\Hom(W,W((z^{\frac{1}{T}}))),\quad
v\mapsto Y_{W}(v,z)=\sum_{n\in\Z}v_{\frac{n}{T}}z^{-\frac{n}{T}-1}
\end{eqnarray*}
such that
 $Y_{W}(\mathbbm{1},z)=\mathrm{Id}_{W}$
and for $u,v\in V$
\begin{eqnarray}
&&z_{0}^{-1}\delta\left(\frac{z_{1}-z_{2}}{z_{0}}\right)Y_{W}(u,z_{1})Y_{W}(v,z_{2})-z_{0}^{-1}
\delta\left(\frac{z_{2}-z_{1}}{-z_{0}}\right)Y_{W}(v,z_{2})Y_{W}(u,z_{1})\nno\\
\label{twisted-jacobi}&&\hspace{1.5cm}=z_{2}^{-1}\sum_{i=0}^{T-1}\frac{1}{T}
\delta\left(\left(\frac{z_{1}-z_{0}}{z_{2}}\right)^{\frac{1}{T}}\omega^{i}\right)Y_{W}(Y(\sigma^{i}(u),z_{0})v,z_{2}).
\end{eqnarray}}
\end{definition}

Then we have the following result (\cite{KW,Li2}).

\begin{thm}\label{re-m}
 Let $\mathcal{C}$ be a  Lie conformal algebra equipped with a finite order automorphism
 $\sigma$. Then any restricted $\widehat{\mathcal{C}}[\sigma]$-module $W$ is naturally a $\sigma$-twisted $V_{\widehat{\mathcal{C}}}$-module with $Y_{W}(a,z)=a^{\sigma}(z)$ for $a\in\CC$. On the other hand, any $\sigma$-twisted $V_{\widehat{\mathcal{C}}}$-module $W$ is naturally a restricted $\widehat{\mathcal{C}}[\sigma]$-module with $a^{\sigma}(z)=Y_{W}(a,z)$ for $a\in \CC$.
\end{thm}

Now we state a graded version of Theorem \ref{re-m}.
 A {\em $\Z$-graded Lie conformal  algebra} is a Lie  conformal algebra $\CC$ equipped with a
 $\Z$-grading
$\CC=\oplus_{n\in\Z}\CC_{(n)}$
such that
$$a_{i}\,\CC_{(n)}\subseteq \CC_{(m+n-i-1)}\quad\text{ for }a\in\CC_{(m)},\ i\in\N.$$
We have the following straightforward result from \eqref{lb}.

\begin{lem}\label{c-graded}
Let $\CC$  be a  Lie conformal algebra equipped with a $\Z$-grading
 $\CC=\oplus_{n\in\Z}\CC_{(n)}$ (of vector spaces).
Then
$\CC$ is a $\Z$-graded  Lie conformal algebra  if and only if
$$
\widehat{\CC}=\oplus_{n\in\Z}\,\widehat{\CC}_{(n)}
$$ is a $\Z$-graded Lie algebra,
where $\widehat{\CC}_{(n)}=\mathrm{Span}\{a(m-n-1)\mid a\in \CC_{(m)},m\in \Z\}$.
Furthermore, in this case, an automorphism $\sigma$ of $\CC$ is graded if and only if the automorphism $\wh\sigma$ of
$\wh{\CC}$ is graded.
\end{lem}

Let  $\sigma$ be an order $T$ graded automorphism of  a $\Z$-graded Lie conformal  algebra $\CC=\oplus_{n\in\mathbb{Z}}{\CC}_{(n)}$, i.e.  $\sigma$ is an automorphism such that $\sigma({\CC}_{(n)})={\CC}_{(n)}$ for $n\in \mathbb{Z}$.
It is obvious that
$\widehat{\mathcal{C}}[\sigma]$ is a $\frac{1}{T}\Z$-graded Lie algebra with
\begin{align}\label{zgradingtw}
\deg a\left(\frac{r}{T}\right)=n-1-\frac{r}{T} \quad\te{for } a\in \CC_{(n)}\cap \CC^{(r)},\ n,r\in \Z.
\end{align}

Recall that a {\em $\Z$-graded vertex algebra} is a vertex algebra  $V$ equipped with a $\Z$-grading $V=\op_{n\in\Z}V_{(n)}$ such that
$$u_{k}\, V_{(n)}\subseteq V_{(m+n-k-1)}\quad \text{ for }u\in V_{(m)},\ m,n,k\in\Z.$$
Then
$V_{\widehat{\CC}}$ is naturally a $\Z$-graded vertex algebra and $\sigma$ is a graded automorphism of it, where the grading on $V_{\widehat{\CC}}$ is
given by
 \begin{equation}
 \begin{split}\label{VCZgrading}
&\deg(a^{1}(-n_{1})\cdots a^{k}(-n_{k})\mathbbm{1})=(m_{1}+n_{1}-1)+\cdots+(m_{k}+n_{k}-1),
 \end{split}
 \end{equation}
for $a^{i}\in\CC_{(m_{i})},\, m_{i}\in\Z$ and  $n_{i},\,k\in\N$.
We say that a $\sigma$-twisted $V_{\widehat{\CC}}$-module $W$ is {\em $\frac{1}{T}\Z$-graded} if
 there is a $\Z$-grading
$W=\oplus_{n\in\Z}W_{(\frac{n}{T})}$
such that
 $$u_{\frac{r}{T}}W_{(\frac{n}{T})}\subseteq W_{(\frac{n}{T}+m-\frac{r}{T}-1)}\quad
\text{for}\  u\in V_{(m)},\ n,r,m\in\Z.$$

Then we have the following graded version of Theorem \ref{re-m}, whose proof is straightforward.

\begin{prpt}\label{prop:gversion}
Let $\mathcal{C}$ be a $\Z$-graded Lie conformal algebra, and let $\sigma$ be a graded automorphism
of $\CC$ with finite order $T$.
Then $\frac{1}{T}\Z$-graded  restricted $\widehat{\mathcal{C}}[\sigma]$-modules
are exactly $\frac{1}{T}\Z$-graded $\sigma$-twisted $V_{\widehat{\mathcal{C}}}$-modules.
\end{prpt}


\section{Associating $\widehat{\mathcal{G}}(\mathcal{J})$ with vertex algebra and twisted modules}
In this section,
 we associate the restricted modules for $\widehat{\mathcal{G}}(\mathcal{J})$  with the twisted modules for the
  vertex algebra arising from a toroidal Lie algebra.

First we recall the $2$-toroidal Lie algebra of type $C_2$. Write $\mathfrak{sp}_4$ for the symplectic Lie algebra of rank $2$, i.e.,
the simple Lie algebra of type $C_2$.
We fix a standard basis of $\mathfrak{sp}_4$ as follows:
\[\{E_{1,1}-E_{3,3},E_{2,2}-E_{4,4},E_{1,3},E_{3,1},E_{2,4},E_{4,2},E_{1,2}-E_{4,3},E_{2,1}-E_{3,4},E_{1,4}+E_{2,3},E_{4,1}+E_{3,2}\}.
\]
Here, $E_{i,j}$ is  the elementary $4\times 4$-matrix having $1$ in $(i,j)$-entry and $0$ elsewhere.

 Let $\mathcal{K}$ be the vector space spanned by the symbols
$$t_{1}^{m_{1}}t_{2}^{m_{2}}\mk_{1},\quad t_{1}^{m_{1}}t_{2}^{m_{2}}\mk_{2} \quad (m_{1},m_2\in\Z),$$ and subject to the relations
 \begin{equation}\label{rela:1}
m_{1}t_{1}^{m_{1}}t_{2}^{m_{2}}\mathbf{k}_{1}+m_{2}t_{1}^{m_{1}}t_{2}^{m_{2}}\mathbf{k}_{2}=0 \quad  \text{for}\ m_1,m_2,n_1,n_2\in \Z.
\end{equation}

 Let
 \[\mathfrak{t}=\left(\mathfrak {sp}_{4}\otimes\mathbb{C}[t_{1}^{\pm1},t_{2}^{\pm1}]\right)\oplus\mathcal{K}\]
  be the $2$-toroidal Lie algebra associated to $\mathfrak {sp}_{4}$, where $\mathcal{K}$ is central and
\begin{eqnarray*}
[x\ot t_{1}^{m_{1}}t_{2}^{m_{2}},y\ot t_{1}^{n_{1}}t_{2}^{n_{2}}]=[x,y]\ot t_{1}^{m_{1}+n_{1}}t_{2}^{m_{2}+n_{2}}+\mathrm{tr}(xy)\sum_{i=1,\,2}m_{i}t_{1}^{m_{1}+n_{1}}t_{2}^{m_{2}+n_{2}}\mathbf{k}_{i}
\end{eqnarray*}
for $x,y\in \mathfrak {sp}_{4}$ and $m_{1},m_{2},n_{1},n_{2}\in\mathbb{Z}$.
It is known  that $\mathfrak{t}$ is the universal central extension of the $2$-loop algebra $\mathfrak {sp}_{4}\otimes\mathbb{C}[t_{1}^{\pm1},t_{2}^{\pm1}]$ (\cite{MRY}).
We fix a  $\Z$-grading on
 $\mathfrak{t}$  with
\begin{align}\label{Zgradingont}
\deg(x\ot t_1^{m_1}t_2^{m_2})= \deg(t_1^{m_1}t_2^{m_2}\mathbf{k}_i)=-m_1
\end{align}
for $x\in \mathfrak {sp}_{4},\, m_{i}\in\mathbb{Z}$ and $i=1,2$.

Form the vector space
\begin{align*}
\mathcal{C}_{\g}=(\C[D]\otimes\g)\oplus\C\mathbf{k}_{1},
\end{align*}
where $D$ is an indeterminate and
\[\g=(\mathfrak {sp}_{4}\otimes\C[t_{2},t_{2}^{-1}])\oplus\C \mathbf{k}_{2}\oplus\left(\sum_{s\in\Z^{\times}}\C t_{2}^{s}\mathbf{k}_{1}\right)\subseteq\mathfrak{t}.\]
Define a linear operator $\partial$ on $\mathcal{C}_{\g}$  by
  $$\partial(D^{n}\otimes a)=D^{n+1}\otimes a\quad\text{and}\quad \partial(\mathbf{k}_{1})=0 \quad \text{for}\ a\in\g,\ n\in\N, $$
and define
\begin{equation*}
Y_{-}(\cdot,z):\CC_\g  \rightarrow\Hom(\CC_\g,z^{-1}\mathcal{C}_{\g}[z^{-1}]),\quad
a\mapsto \sum_{n\ge 0}a_{n}z^{-n-1}
\end{equation*}
to be the unique linear map such that the following conditions hold:
\begin{itemize}
\item $[\partial,Y_-(a,z)]=\frac{d}{dz}Y_{-}(a,z)=Y_{-}(\partial a,z)$,
\item $Y_{-}(\mathbf{k}_1,z)=Y_{-}(\mathbf{k}_{2},z)= Y_{-}(t_{2}^{s}\mathbf{k}_{1},z)=0$, and
\item
\begin{equation*}
\begin{split}
&\qquad(x\ot t_{2}^{m})_{n}(y\ot t_{2}^{m'})\\
&=\left\{\begin{array}{lll}
    \ [x,y]\ot t_{2}^{m+m'}+\frac{m}{m+m'}\text{tr}(xy)(D\otimes t_{2}^{m+m'}\mathbf{k}_{1}), & n=0,\  m+m'\neq0,\\
    \ [x,y]\ot 1+m\text{tr}(xy)\mathbf{k}_{2}, &  n=0,\ m+m'=0, \\
  \ \text{tr}(xy) t_{2}^{m+m'}\mathbf{k}_{1}, & n=1,\\
    \ 0, & n>1,
  \end{array}
    \right.
\end{split}\end{equation*}
\end{itemize}
where $a\in \g,\ x,y\in \fsp_{4},\ s\in\Z^{\times},\ m,m'\in\Z$ and $n\in\N$.

The following result was proved in \cite{CJKT}.
\begin{lem}\label{lem:cg}
The vector space $\CC_{\g}$ together with the operator $\partial$ and the map $Y_-$ forms a Lie conformal algebra. Moreover, the linear map
$i_{\g}: \widehat{\mathcal{C}_{\g}}\rightarrow \mathfrak{t}$ defined by
\begin{equation*}
\begin{split}
&(x\ot t_{2}^{n})(m)\mapsto x\ot t_{1}^{m}t_{2}^{n},\ \ \mathbf{k}_{2}(m)\mapsto t_{1}^{m} \mathbf{k}_{2},\ \
(t_{2}^{s}\mathbf{k}_{1})(m)\mapsto t_{1}^{m+1}t_{2}^{s}\mathbf{k}_{1},\  \mk_{1}(-1)\mapsto \mk_{1}
\end{split}\end{equation*}
for $x\in\fsp_{4},\ n,m\in\Z$ and $s\in\Z^{\times}$, is a Lie isomorphism.
\end{lem}


Endow a $\Z$-grading (of vector spaces) on $\CC_{\g}$
such that $\mathcal{C}_{\g (n)}=0$ if $n<0$,
 $\mathcal{C}_{\g (0)}=\te{Span}\{t_{2}^{n}\mk_{1}\mid n\in\Z\}$ and
\begin{equation*}
\mathcal{C}_{\g (n)}=\te{Span}\{D^{n-1}\ot a,\ D^{n}\ot t_{2}^{s}\mk_{1}\mid a\in (\fsp_{4}\ot\C[t_{2},t_{2}^{-1}])\op\C\mk_{2},\ s\in\Z^{\times}\}
  \end{equation*}
 if $n>0$.
  This induces a $\Z$-grading (of vector spaces) on $\widehat{\CC_{\g}}$ such that
  \begin{align}\label{hatcg-graded}
\widehat{\CC_{\g}}=\oplus_{n\in\Z}\ \widehat{\CC}_{\g(n)},\end{align}
   where $\widehat{\CC_{\g}}_{(n)}=\te{Span}\{a(m-n-1)\mid m\in\N, a\in\CC_{\g(m)}\}$.

Via the isomorphism $i_\g$ (see Lemma \ref{lem:cg}), we see that
 the $\Z$-grading on $\wh{\CC_{\g}}$ coincides with that of $\mathfrak{t}$ defined in \eqref{Zgradingont}.
 Thus $\widehat{\CC_{\g}}$ is a $\Z$-graded Lie algebra (with the grading \eqref{hatcg-graded}).
  This together with Lemma \ref{c-graded} implies that
  \begin{eqnarray}\label{cg-graded}
\mathcal{C}_{\g}=\oplus_{n\in\Z}\ \mathcal{C}_{\g (n)}
\end{eqnarray} is a $\Z$-graded Lie conformal  algebra.

In what follows we define a graded involution of the Lie conformal algebra $\CC_{\g}$.
Firstly, let  $\sigma:\g\rightarrow\mathcal{C}_{\g}$ be a linear map  defined by
\begin{equation*}\begin{split}
&\sigma((E_{1,1}-E_{3,3})\ot t_{2}^{n})=\left\{\begin{array}{ll}           (E_{4,4}-E_{2,2})\ot 1+\mk_2,\ &n=0,\\       (E_{4,4}-E_{2,2})\ot t_2^n+\frac{1}{n}D\ot t^n_2\mk_1,\ &n\neq0, \\          \end{array}
            \right.\\
&\sigma((E_{2,2}-E_{4,4})\ot t_{2}^{n})=\left\{\begin{array}{ll}
                        (E_{3,3}-E_{1,1})\ot 1+\mk_2,\ &n=0,\\
                         (E_{3,3}-E_{1,1})\ot t_2^n+\frac{1}{n}D\ot t^n_2\mk_1,\ &n\neq0,\\
             \end{array}
           \right.\\    &\sigma((E_{1,2}-E_{4,3})\ot t^n_2)=(E_{1,2}-E_{4,3})\ot t^n_2 ,\quad \sigma((E_{1,4}+E_{2,3})\ot t^n_2)=-(E_{4,1}+E_{3,2})\ot t^{n+1}_2,\\
    &\sigma((E_{2,1}-E_{3,4})\ot t^n_2)=(E_{2,1}-E_{3,4})\ot t^n_2 ,\quad \sigma((E_{4,1}+E_{3,2})\ot t^n_2)=-(E_{1,4}+E_{2,3})\ot t^{n-1}_2, \\
&\sigma(E_{1,3}\ot t^n_2)=E_{4,2}\ot t^{n+1}_2,\quad \sigma(E_{3,1}\ot t^n_2)=E_{2,4}\ot t^{n-1}_2,\quad \sigma(E_{2,4}\ot t^n_2)=E_{3,1}\ot t^{n+1}_2,\\
&\sigma(E_{4,2}\ot t^n_2)=E_{1,3}\ot t^{n-1}_2,\quad \sigma(\mathbf{k}_{2})=\mathbf{k}_{2} ,\quad \sigma(t_{2}^{s}\mathbf{k}_{1})=t_{2}^{s}\mathbf{k}_{1},
\end{split}\end{equation*}
where $n\in\Z,\ s\in\Z^{\times}$.
Next, we extend $\sigma$ to a linear endomorphism on $\mathcal{C}_{\g}$ such that
$$\sigma(D^{n}\ot a)=\partial^{n}(\sigma(a))\quad\te{and}\quad\sigma(\mathbf{k}_{1})=\mathbf{k}_{1} \quad \text{for}\ a\in\g,\ n\in\N.$$

 We have the following result.

\begin{lem}\label{lem:sigmaaut}
The map $\sigma$ is a graded involution of  $\CC_{\g}$.
\end{lem}
\begin{proof} It is straightforward to check that $\sigma^2=\mathrm{Id}$ and $\partial\circ\sigma=\sigma\circ\partial$.
By Lemma \ref{au-cla} and Lemma \ref{c-graded}, it remains to prove that the map
$$\widehat{\sigma}: \widehat{\mathcal{C}_{\g}}\longrightarrow  \widehat{\mathcal{C}_{\g}}, \quad \widehat{\sigma}(a(m))=\sigma(a)(m)\quad (a\in\mathcal{C}_{\g},\ m\in\Z)$$
is a graded Lie automorphism of $\widehat{\CC_{\g}}$. From
  Lemma \ref{lem:cg},  we only need to show that
 $i_{\g}\circ\widehat{\sigma}\circ i_{\g}^{-1}$ is a graded automorphism of $\mathfrak{t}$.

This can be checked directly by using the following
explicit action of $i_{\g}\circ\widehat{\sigma}\circ i_{\g}^{-1}$ on $\mathfrak{t}$:
\begin{eqnarray*}
&&\mk_{1}\mapsto \mk_{1},\ \ \ t_{1}^{m}\mathbf{k}_{2}\mapsto t_{1}^{m} \mathbf{k}_{2},\ \ \  t_{1}^{m}t_{2}^{s}\mathbf{k}_{1}\mapsto t_{1}^{m}t_{2}^{s}\mathbf{k}_{1}, \\
&&(E_{1,2}-E_{4,3})\ot t_{1}^{m}t^n_2\mapsto (E_{1,2}-E_{4,3})\ot t_{1}^{m}t^{n}_2,\\
&&(E_{2,1}-E_{3,4})\ot t_{1}^{m}t^n_2\mapsto (E_{2,1}-E_{3,4})\ot t_{1}^{m}t^{n}_2,\\
&&(E_{1,4}+E_{2,3})\ot t_{1}^{m}t^n_2\mapsto-(E_{4,1}+E_{3,2})\ot t_{1}^{m}t^{n+1}_2,\\
 && (E_{4,1}+E_{3,2})\ot t_{1}^{m}t^n_2\mapsto -(E_{1,4}+E_{2,3})\ot t_{1}^{m}t^{n-1}_2,\\
 &&(E_{1,1}-E_{3,3})\ot t_{1}^{m}t^n_2\mapsto (E_{4,4}-E_{2,2})\ot t_{1}^{m}t^{n}_2+t_{1}^{m}t_{2}^{n}\mathbf{k}_{2},\\
       &&(E_{2,2}-E_{4,4})\ot t_{1}^{m}t^n_2\mapsto (E_{3,3}-E_{1,1})\ot t_{1}^{m}t^{n}_2+t_{1}^{m}t_{2}^{n}\mathbf{k}_{2},\\
 && E_{1,3}\ot t_{1}^{m}t^n_2\mapsto E_{4,2}\ot t_{1}^{m}t^{n+1}_2,\quad E_{2,4}\ot t_{1}^{m}t^n_2\mapsto E_{3,1}\ot t_{1}^{m}t^{n+1}_2,\\
  && E_{3,1}\ot t_{1}^{m}t^n_2\mapsto E_{2,4}\ot t_{1}^{m}t^{n-1}_2,\quad E_{4,2}\ot t_{1}^{m}t^n_2\mapsto E_{1,3}\ot t_{1}^{m}t^{n-1}_2,
         \end{eqnarray*}
where $m,n\in \Z,\ s\in \Z^{\times}$.
\end{proof}

In view of Lemma \ref{lem:sigmaaut}, associated to the pair $(\CC_{\g},\sigma)$,
we have a Lie algebra
$$\widehat{\mathcal{C}}_{\g}[\sigma]=\bigoplus_{n\in\Z}\left(\CC_{\g}^{(n)}\otimes t_1^{\frac{n}{2}}\right)/\te{Image}\big(\partial\otimes1+1\otimes \frac{d}{dt}\big),$$
where $\CC_{\g}^{(n)}=\{a\in\CC_{\g}\mid\sigma(a)=(-1)^{n}a\}$ and the Lie brackets  are given by \eqref{t-lb}.
Note that the following elements form a basis of $\widehat{\mathcal{C}}_{\g}[\sigma]$:
\begin{equation*}\begin{split}
&\left((E_{1,3}\ot t_{2}^{n})^{(j)}\right)\lf(m+\frac{j}{2}\rg),\hspace{2cm} \left((E_{3,1}\ot t_{2}^{n})^{(j)}\right)\lf(m+\frac{j}{2}\rg),\\
&\left(((E_{1,2}-E_{4,3})\ot t_{2}^{n})^{(0)}\right)(m),\hspace{1.6cm}\left(((E_{2,1}-E_{3,4})\ot t_{2}^{n})^{(0)}\right)(m),\\
&\left(((E_{1,1}-E_{3,3})\ot t_{2}^{n})^{(j)}\right)\lf(m+\frac{j}{2}\rg),\quad \left(((E_{1,4}+E_{2,3})\ot t_{2}^{n})^{(j)}\right)\lf(m+\frac{j}{2}\rg),\\
&\,\mathbf{k}_{2}(m),\quad (t_{2}^{s}\mathbf{k}_{1})(m),\quad\mathbf{k}_{1}(-1)
\end{split}\end{equation*}
for $j\in\{0,1\},\ n,m\in\Z$ and  $s\in \Z^{\times}$, where $a^{(j)}=\frac{1}{2}(a+(-1)^{j}\sigma(a))\in\CC_{\g}^{(j)}$ for $a\in\CC_{\g}$.

Define a linear map $\varphi: \widehat{\mathcal{G}}(\mathcal{J})\rightarrow \widehat{\mathcal{C}}_{\g}[\sigma]$ as follows:
\begin{eqnarray*}
&&x_{+}(2m+j,2n+k)\mapsto \left\{\begin{array}{lll}
 2\left((E_{1,3}\ot t_{2}^{n})^{(j)}\right)\left(m+\frac{j}{2}\right),&\quad \  k=0 ,\ j\in\{0,1\},\\
    \sqrt{-1}\left(((E_{1,2}-E_{4,3})\ot t_{2}^{n+1})^{(0)}\right)(m),&\quad \ k=1,\ j=0,\\
                         \end{array}
                         \right.\\
                         \\
&&x_{-}(2m+j,2n+k)\mapsto \left\{\begin{array}{lll}
 2\left((E_{3,1}\ot t_{2}^{n})^{(j)}\right)\lf(m+\frac{j}{2}\rg), &\quad \ \ k=0 ,\ j\in\{0,1\}, \\  -\sqrt{-1}\left(((E_{2,1}-E_{3,4})\ot t_{2}^{n})^{(0)}\right)(m),&\quad \ \ k=1,\ j=0,\\
                         \end{array}
                         \right.\\
                         \\
&&\alpha^{\vee}(2m+j,2n+k)\mapsto \left\{\begin{array}{lll}
2\left(((E_{1,1}-E_{3,3})\ot t_{2}^{n})^{(j)}\right)(m+\frac{j}{2}), &k=0 ,\ j\in\{0,1\},\\
                           2\sqrt{-1}\left(((E_{1,4}+E_{2,3})\ot t_{2}^{n})^{(j)}\right)\lf(m+\frac{j}{2}\rg) ,&k=1 ,\ j\in\{0,1\},\\
                         \end{array}
                         \right.\\
                         \\
                         &&C_1(2m,2n)\mapsto \frac{1}{2}t_{2}^{n}\mathbf{k}_1(m-1),\ C_2(2m,0)\mapsto \frac{1}{2}\mathbf{k}_{2}(m),\
     C_2(2m,2s)\mapsto -\frac{m}{2s}t_{2}^{s}\mathbf{k}_1(m-1),
\end{eqnarray*}
where $n,m\in\mathbb{Z}$ and $s\in \mathbb{Z}^{\times}$.

The following result affords a new realization of $\widehat{\mathcal{G}}(\mathcal{J})$ in the term of Lie conformal algebra.

\begin{prpt}\label{prop:iso-Cg-tkk}
The map $\varphi: \widehat{\mathcal{G}}(\mathcal{J}) \rightarrow \widehat{\mathcal{C}}_{\g}[\sigma]$  is a Lie isomorphism.
\end{prpt}
\begin{proof}
Clearly, $\varphi$ is a linear isomorphism. Then
we only need to check that $\varphi$ is a Lie homomorphism, i.e., $\varphi$ preserves relations $\text{(R1)-(R4)}$.
Set $\rho=(2m+j)e_1+(2n+k)e_2$ and $\tau=(2m'+j')e_1+(2n'+k')e_2$, where $m,m',n,n'\in \Z$ and $j,k,j',k'\in \{0,1\}$.

For $\text{(R1)}$, we divide the argument into four cases: (i) $\rho\notin S,\,\tau\in S$,
    (ii) $\rho,\, \tau \notin S$,
  (iii) $\rho, \tau\in S\  \text{and}\ \rho+\tau \in S$,
   (iv) $\rho, \tau\in S\  \text{and}\ \rho+\tau \notin S$.

We prove the case (i), and other cases can be proved similarly. For (i), we divide it into two subcases.

\textbf{Subcase 1.}  If $\rho\notin S,\,\tau\in S_{0}\cup S_{1}$, i.e., $j=k=1,\ k'=0$, then $\Omega(\tau)=\frac{1}{2}((-1)^{j'}-1)$. We have
\begin{align*}
\begin{split}
 &\qquad[\varphi(\alpha^{\vee}(\rho)),\varphi(\alpha^{\vee}(\tau))]
\\
 &=4\sqrt{-1}\left[\left(((E_{1,4}+E_{2,3})\ot t_{2}^{n})^{(1)}\right)\lf(m+\frac{1}{2}\rg),
 \left(((E_{1,1}-E_{3,3})\ot t_{2}^{n'})^{(j')}\right)\lf(m'+\frac{j'}{2}\rg)\right]\\
&=\sqrt{-1}\sum_{l\in \N}\binom{m+\frac{1}{2}}{l}\left( (E_{1,4}+E_{2,3})\ot t_{2}^{n}+(E_{4,1}+E_{3,2})\ot t_{2}^{n+1}\right)_{l}\\
 &\left((E_{1,1}-E_{3,3}+(-1)^{j'}(E_{4,4}-E_{2,2}))\ot t_{2}^{n'}\right)\lf(m+m'+\frac{j'+1}{2}-l\rg)\\
 & = 2\sqrt{-1}((-1)^{j'}-1) \left(((E_{1,4}+E_{2,3})\ot t_{2}^{n+n'})^{(0)}\right)\lf(m+m'+\frac{j'+1}{2}\rg)
 \\
 &=2\Omega(\tau)\varphi(\alpha^{\vee}(\rho+\tau)).
\end{split}\end{align*}

\textbf{Subcase 2.} If $\rho\notin S,\, \tau\in S_{2}$, i.e., $
j=k=k'=1,\ j'=0$, then $\Omega(\tau)=1$.
We have
\begin{align*}
\begin{split}
 &\qquad[\varphi(\alpha^{\vee}(\rho)),\varphi(\alpha^{\vee}(\tau))]
\\
 &=-4\left[\left(((E_{1,4}+E_{2,3})\ot t_{2}^{n})^{(1)}\right)\lf(m+\frac{1}{2}\rg),
 \left(((E_{1,4}+E_{2,3})\ot t_{2}^{n'})^{(0)}\right)(m') \right]\\
&=-\sum_{l\in \N}\binom{m+\frac{1}{2}}{l}\left( (E_{1,4}+E_{2,3})\ot t_{2}^{n}+(E_{4,1}+E_{3,2})\ot t_{2}^{n+1}\right)_{l}\\
 &\left((E_{1,4}+E_{2,3})\ot t_{2}^{n'}-(E_{4,1}+E_{3,2})\ot t_{2}^{n'+1}
 \right)\lf(m+m'+\frac{1}{2}-l\rg)\\
 &=\left\{\begin{array}{lll}
2\left((E_{1,1}-E_{3,3}+E_{2,2}-E_{4,4})\ot 1 -\mk_{2}\right)\lf(m+m'+\frac{1}{2}\rg),&
\qquad n+n'+1= 0,\\
& \\
2((E_{1,1}-E_{3,3}+E_{2,2}-E_{4,4})\ot t_{2}^{n+n'+1})\lf(m+m'+\frac{1}{2}\rg)&\\
-\frac{2}{n+n'+1}(D\ot t_2^{n+n'+1}\mk_1)\lf(m+m'+\frac{1}{2}\rg),&\qquad n+n'+1\neq 0, \\
  \end{array}
    \right.
    \end{split}\end{align*}
    and
    \begin{align*}
    \begin{split}
        &\qquad\varphi(\alpha^{\vee}(\rho+\tau))\\
        &=\left(((E_{1,1}-E_{3,3})\ot t_{2}^{n+n'+1})^{(1)}\right)\lf(m+m'+\frac{1}{2}\rg)\\
        &=\left\{\begin{array}{lll}
\left((E_{1,1}-E_{3,3}+E_{2,2}-E_{4,4})\ot 1 -\mk_{2}\right)\lf(m+m'+\frac{1}{2}\rg),&
\quad n+n'+1= 0,\\
& \\
((E_{1,1}-E_{3,3}+E_{2,2}-E_{4,4})\ot t_{2}^{n+n'+1})\lf(m+m'+\frac{1}{2}\rg)&\\
-\frac{1}{n+n'+1}(D\ot t_2^{n+n'+1}\mk_1)\lf(m+m'+\frac{1}{2}\rg),&\quad n+n'+1\neq 0. \\
  \end{array}
    \right.
        \end{split}
    \end{align*}
    This shows $\text{(R1)}$.

    For $\text{(R2)}$, we divide the argument into two cases:
(i) $ \rho\in S,\,\tau\in S$,   (ii) $\rho\notin S,\,\tau\in S$. We prove the case (ii),  and the case (i) can be proved similarly. For (ii), we divide it into two subcases.

  \textbf{Subcase 1.}  If $\rho\notin S,\, \tau\in S_{0}\cup S_{1}$, i.e., $j=k=1,\ k'=0$, then $\Omega(\tau)=\frac{1}{2}((-1)^{j'}-1)$. We have
\begin{align*}
\begin{split}
 &\qquad[\varphi(\alpha^{\vee}(\rho)),\varphi(x_{+}(\tau))]
\\
 &=4\sqrt{-1}\left[\left(((E_{1,4}+E_{2,3})\ot t_{2}^{n})^{(1)}\right)\left(m+\frac{1}{2}\right),\left((E_{1,3}\ot t_{2}^{n'})^{(j')}\right)\left(m'+\frac{j'}{2}\right)\right]\\
&=\sqrt{-1}\sum_{l\in \N}\binom{m+\frac{1}{2}}{l}\left( (E_{1,4}+E_{2,3})\ot t_{2}^{n}+(E_{4,1}+E_{3,2})\ot t_{2}^{n+1}\right)_{l}\\
&\left( E_{1,3}\ot t_{2}^{n'}+(-1)^{j'}E_{4,2}\ot t_{2}^{n'+1}\right)\lf(m+m'+\frac{j'+1}{2}-l\rg)\\
&=\sqrt{-1}((-1)^{j'}-1)\left(((E_{1,2}-E_{4,3})\ot t_{2}^{n+n'+1})^{(0)}\right)\left(m+m'+\frac{j'+1}{2}\right)\\
 &=2\Omega(\tau)\varphi(x_{+}(\rho+\tau)),
\end{split}\end{align*}
and

\begin{equation*}
\begin{split}
&\qquad[\varphi(\alpha^{\vee}(\rho)),\varphi(x_{-}(\tau))]
\\
 &=4\sqrt{-1}\left[\left(((E_{1,4}+E_{2,3})\ot t_{2}^{n})^{(1)}\right)\left(m+\frac{1}{2}\right),\left((E_{3,1}\ot t_{2}^{n'})^{(j')}\right)\left(m'+\frac{j'}{2}\right)\right]\\
&=\sqrt{-1}\sum_{l\in \N}\binom{m+\frac{1}{2}}{l}\left( (E_{1,4}+E_{2,3})\ot t_{2}^{n}+(E_{4,1}+E_{3,2})\ot t_{2}^{n+1}\right)_{l}\\
&\left( E_{3,1}\ot t_{2}^{n'}+(-1)^{j'}E_{2,4}\ot t_{2}^{n'-1}\right)\left(m+m'+\frac{j'+1}{2}-l\right)\\
&=\sqrt{-1}((-1)^{j'}-1)\left(((E_{3,4}-E_{2,1})\ot t_{2}^{n+n'})^{(0)}\right)\left(m+m'+\frac{j'+1}{2}\right)\\
 &=2\Omega(\tau)\varphi(x_{-}(\rho+\tau)).
\end{split}\end{equation*}

    \textbf{Subcase 2.}  If $\rho\notin S,\, \tau\in S_{2}$, i.e., $j=k=k'=1,\ j'=0$, then
    $\Omega(\tau)=1$. We have
\begin{equation*}
\begin{split}
 &\qquad\left[\varphi(\alpha^{\vee}(\rho)),\varphi(x_{+}(\tau))\right] \\
 &=-2\left[\left(((E_{1,4}+E_{2,3})\ot t_{2}^{n})^{(1)}\right)\left(m+\frac{1}{2}\right),\left(((E_{1,2}-E_{4,3})\ot t_{2}^{n'+1})^{(0)}\right)(m')\right]\\
&=-\sum_{l\in \N}\binom{m+\frac{1}{2}}{l}\left( (E_{1,4}+E_{2,3})\ot t_{2}^{n}+(E_{4,1}+E_{3,2})\ot t_{2}^{n+1}\right)
_{l}\\
&\left((E_{1,2}-E_{4,3})\ot t_{2}^{n'+1}\right)\left(m+m'+\frac{1}{2}-l\right)\\
&=4\left((E_{1,3}\ot t_{2}^{n+n'+1})^{(1)}\right)\left(m+m'+\frac{1}{2}\right)\\
&=2\Omega(\tau)\varphi(x_{+}(\rho+\tau)),\\
\end{split}
\end{equation*}
and
\begin{align*}
\begin{split}
 &\qquad\left[\varphi(\alpha^{\vee}(\rho)),\varphi(x_{-}(\tau)\right] \\
 &=2\left[\left(((E_{1,4}+E_{2,3})\ot t_{2}^{n})^{(1)}\right)\left(m+\frac{1}{2}\right),\left(((E_{2,1}-E_{3,4})\ot t_{2}^{n'})^{(0)}\right)(m')\right]\\
&=\sum_{l\in \N}\binom{m+\frac{1}{2}}{l}\left( (E_{1,4}+E_{2,3})\ot t_{2}^{n}+(E_{4,1}+E_{3,2})\ot t_{2}^{n+1}\right)
_{l}\\
&\left((E_{2,1}-E_{3,4})\ot t_{2}^{n'}\right)\left(m+m'+\frac{1}{2}-l\right)\\
&=4\left((E_{3,1}\ot t_{2}^{n+n'+1})^{(1)}\right)\left(m+m'+\frac{1}{2}\right)\\
&=2\Omega(\tau)\varphi(x_{-}(\rho+\tau)).
\end{split}
\end{align*}

For $\text{(R3)}$, we have
\begin{equation*}
\begin{split}
  &\left((E_{1,3}\ot t_{2}^n)^{(j)}\right)_{l}\left((E_{1,3}\ot t_{2}^{n'}
  )^{(j')}\right)=0,\\
 &\left((E_{1,3}\ot t_{2}^n)^{(j)}\right)_{l}\left(
  ((E_{1,2}-E_{4,3})
  \ot t_{2}^{n'})^{(0)}\right)=0,\\
& \left(
 ((E_{1,2}-E_{4,3})
  \ot t_{2}^{n})^{(0)}\right)_{l}\left(
 ((E_{1,2}-E_{4,3})
  \ot t_{2}^{n'})^{(0)}\right)=0
\end{split}
\end{equation*}
for all $l\in \N$. This implies that  $[x_{+}(\rho),x_{+}(\tau)]=0$. Similarly we have that $[x_{-}(\rho),x_{-}(\tau)]=0$.
Thus the first identity is shown.

For the second identity, we divide the proof into two cases:
(i) $\rho,\tau\in S \ \text{and} \ \rho+\tau\notin S$,  (ii) $ \rho,\tau\in S    \ \text{and} \ \rho+\tau\in S$.

For (i), we divide it into two subcases. The case (ii) can be proved similarly.

\textbf{Subcase 1.}  If $\rho\in S_{1},\,\tau\in S_{2}$, i.e., $j=k'=1,\ j'=k=0$,  then $\Omega(\tau)=1$. We have
\begin{equation*}
\begin{split}
&\qquad\lf[\varphi(x_{+}(\rho)),\varphi(x_{-}(\tau))\rg]\\
 &=-2\sqrt{-1}\lf[\left((E_{1,3}\ot t_{2}^{n})^{(1)}\right)\lf(m+\frac{1}{2}\rg),\left(((E_{2,1}-E_{3,4})\ot t_{2}^{n'})^{(0)}\right)(m')\rg]\\
&=-\sqrt{-1}\sum_{l\in \N}\binom{m+\frac{1}{2}}{l}\left(E_{1,3}\ot t_{2}^{n}-E_{4,2}\ot t_{2}^{n+1}\right)_{l}\left((E_{2,1}-E_{3,4})\ot t_{2}^{n'}\right)\lf(m+m'+\frac{1}{2}-l\rg)\\
&=2\sqrt{-1}\left(((E_{1,4}+E_{2,3})\ot t_{2}^{n+n'})^{(1)}\right)\lf(m+m'+\frac{1}{2}\rg)\\
&=\Omega(\tau)\varphi(\alpha^{\vee}(\rho+\tau)).\\
\end{split}
\end{equation*}

\textbf{Subcase 2.}  If $\rho\in S_{2},\,\tau\in S_{1}$, i.e., $j=k'=0,\ j'=k=1$,  then $\Omega(\tau)=-1$. We have
\begin{equation*}
\begin{split}
&\qquad\lf[\varphi(x_{+}(\rho)),\varphi(x_{-}(\tau))\rg]\\
 &=2\sqrt{-1}\lf[\left(((E_{1,2}-E_{4,3})\ot t_{2}^{n+1})^{(0)}\right)(m),\left((E_{3,1}\ot t_{2}^{n'})^{(1)}\right)\lf(m'+\frac{1}{2}\rg)\rg]\\
 &=\sqrt{-1}\sum_{l\in \N}\binom{m}{l}\left(
   (E_{1,2}-E_{4,3})\ot t_{2}^{n+1}\right)_{l}\left(E_{3,1}\ot t_{2}^{n'}-E_{2,4}\ot t_{2}^{n'-1} \right)\lf(m+m'+\frac{1}{2}-l\rg)\\
    &= -2\sqrt{-1}\left(((E_{1,4}+E_{2,3})\ot t_{2}^{n+n'})^{(1)}\right)\lf(m+m'+\frac{1}{2}\rg)\\
    &=\Omega(\tau)\varphi(\alpha^{\vee}(\rho+\tau)).
\end{split}
\end{equation*}
This completes the proof of $\text{(R3)}$.

 Finally, $\text{(R4)}$ follows from the Lemma \ref{lem:cg} and the equation \eqref{rela:1}.
Thus $\varphi$ is
a Lie isomorphism from $\widehat{\mathcal{G}}(\mathcal{J})$ to $\widehat{\mathcal{C}}_{\g}[\sigma]$, as desired.
\end{proof}

Recall that $\widehat{\mathcal{C}}_{\g}[\sigma]$ is  a $\frac{1}{2}\Z$-graded Lie algebra with the grading given by \eqref{zgradingtw}.
  On the other hand,
   $\widehat{\mathcal{G}}(\mathcal{J})$ is a $\Z$-graded Lie algebra with the grading given by  \eqref{zgradtkk}.
By definition, we immediately have the following result.

\begin{lem}\label{lem:zvsfrac12z} For every $n\in \Z$, we have that $\varphi(\widehat{\mathcal{G}}(\mathcal{J})_{(n)})=\widehat{\mathcal{C}}_{\g}[\sigma]_{(\frac{n}{2})}$.
\end{lem}

For every $\ell\in\C$, we introduce the quotient $\widehat{\CC_\g}$-module
\begin{align*}V_{\widehat{\mathcal{C}_{\g}}}(\ell,0)=V_{\widehat{\CC_\g}}/K_\ell,
\end{align*}
where
$K_\ell$ is the $\widehat{\CC_\g}$-submodule of $V_{\widehat{\CC_\g}}$ generated by
the element  $\mk_{1}(-1)\mathbbm{1}-\ell$.
Since $\mk_1$ is central in $\widehat{\mathcal{C}_{\g}}$, $K_\ell$ is an ideal of  $V_{\widehat{\CC_\g}}$ (as a vertex algebra).
Then $V_{\widehat{\mathcal{C}_{\g}}}(\ell,0)$ is naturally a vertex algebra.
On the other hand, recall that $\sigma$ can be uniquely extended to an automorphism of $V_{\wh{\CC_{\g}}}$ with order $2$ (see \eqref{extendsigma}).
As $\sigma(\mk_1)=\mk_1$,  $\sigma$ preserves the ideal $K_\ell$, and hence
induces an order $2$ automorphism, still called $\sigma$, of the quotient vertex algebra $V_{\widehat{\CC_{\g}}}(\ell,0)$.

Recall that $V_{\widehat{\CC_\g}}$ is a $\Z$-graded vertex algebra  with the grading given by \eqref{VCZgrading}.
Since the ideal $K_\ell$ is graded, $V_{\widehat{\CC_{\g}}}(\ell,0)$ is a $\Z$-graded vertex algebra as well.
Furthermore, $\sigma$ is a graded automorphism of $V_{\widehat{\CC_{\g}}}(\ell,0)$, as it is graded as an automorphism of $\CC_\g$ (and hence $V_{\wh{\CC_\g}}$).

\begin{definition}{\em A $\widehat{\mathcal{G}}(\mathcal{J})$-module $W$ is said to be {\em of level $\ell\in\C$} if $2C_1$ acts as the scalar $\ell$, and is said to be {\em restricted } if for any $w\in W$, $n\in\mathbb{Z}$ and $s\in\Z^{\times}$,
\begin{eqnarray*}
x_{\pm}(m,n)w=0,\ \alpha^{\vee}(m,n)w=0,\
C_1(2m,2s)w=0,\ C_2(2m,0)w=0
 \end{eqnarray*}
whenever $m$ is sufficiently large.}
\end{definition}

For $a\in\g$, set
 $$\overline{a}^{\sigma}(z)=\sum_{r\in \Z}\varphi^{-1}(a^{(r)}(\frac{r}{2}))
z^{-\frac{r}{2}-1}\in \widehat{\mathcal{G}}[[z^{\frac{1}{2}},z^{-\frac{1}{2}}]],$$
recalling that $a^{(r)}=\frac{1}{2}(a+(-1)^r \sigma(a))$.
Note that a $\widehat{\mathcal{G}}(\mathcal{J})$-module $W$ is restricted if and only if $\overline{a}^{\sigma}(z)\in \mathrm{Hom}(W,W((z^{\frac{1}{2}})))$
for all $a\in \g$.
The following is the first main result of this paper.

\begin{thm}\label{rem} Let $\ell$ be any complex number.
 Every restricted $\widehat{\mathcal{G}}(\mathcal{J})$-module $W$ of level $\ell$ is naturally a $\sigma$-twisted $V_{\widehat{\CC_{\g}}}(\ell,0)$-module with $Y_{W}(a,z)=\overline{a}^{\sigma}(z)$ for $a\in\g$. On the other hand, every $\sigma$-twisted $V_{\widehat{\CC_{\g}}}(\ell,0)$-module $W$ is naturally a restricted $\widehat{\mathcal{G}}(\mathcal{J})$-module of level $\ell$ with $\overline{a}^{\sigma}(z)=Y_{W}(a,z)$ for $a\in\g$.
 Furthermore, $\Z$-graded restricted $\widehat{\mathcal{G}}(\mathcal{J})$-modules of level $\ell$ are exactly $\frac{1}{2}\Z$-graded
 $\sigma$-twisted $V_{\widehat{\CC_{\g}}}(\ell,0)$-modules.
\end{thm}
\begin{proof} Let $W$ be a restricted $\widehat{\mathcal{G}}(\mathcal{J})$-module of level $\ell$. From Proposition \ref{prop:iso-Cg-tkk}, we know that it is a restricted $\widehat{\mathcal{C}}_{\g}[\sigma]$-module with $$a^{\sigma}(z)=\overline{a}^{\sigma}(z)\in \Hom(W,W((z^{\frac{1}{2}})))$$
 for $a\in\g$ and with $\mk_{1}(-1)$ acts as the scalar $\ell$.
  Note that the vertex algebra $V_{\widehat{\mathcal{C_{\g}}}}$ is generated by $\g\op\C\mk_{1}$.
  From Theorem \ref{re-m}, we know that $W$ is a $\sigma$-twisted $V_{\widehat{\mathcal{C_{\g}}}}$-module with $Y_{W}(a,z)=\overline{a}^{\sigma}(z)$ for $a\in\g$.
  Thus $W$ becomes a $\sigma$-twisted $V_{\widehat{\CC_{\g}}}(\ell,0)$-module with $Y_{W}(a,z)=\overline{a}^{\sigma}(z)$ for $a\in\g$.
If $W$ is $\Z$-graded, then by Lemma \ref{lem:zvsfrac12z}, $W$ is $\frac{1}{2}\Z$-graded as a $\widehat{\mathcal{C}}_{\g}[\sigma]$-module.
This together with Proposition \ref{prop:gversion} implies that $W$ is a $\frac{1}{2}\Z$-graded $\sigma$-twisted $V_{\widehat{\CC_{\g}}}(\ell,0)$-module.

On the other hand, let $W$ be a $\sigma$-twisted $V_{\widehat{\CC_{\g}}}(\ell,0)$-module. Then it is a $\sigma$-twisted $V_{\widehat{\mathcal{C_{\g}}}}$-module with $\mk_{1}$ acts as the scalar $\ell$. From Theorem \ref{re-m}, we see that $W$ is a restricted $\widehat{\mathcal{C}}_{\g}[\sigma]$-module with $a^{\sigma}(z)=Y_{W}(a,z)$ for $a\in\g$ and with $\mk_{1}(-1)$ acts as the scalar $\ell$. Thus $W$ is a restricted $\widehat{\mathcal{G}}(\mathcal{J})$-module of level $\ell$ with $\overline{a}^{\sigma}(z)=Y_{W}(a,z)$ for $a\in\g$.
Assume further that $W$ is $\frac{1}{2}\Z$-graded. Again by Proposition \ref{prop:gversion} and Lemma \ref{lem:zvsfrac12z}, we see that
$W$ is a $\Z$-graded  $\widehat{\mathcal{G}}(\mathcal{J})$-module.
\end{proof}

\section{Classification of irreducible graded twisted modules of
 $L_{\widehat{\CC_{\g}}}(\ell,0)$}
 In this section, for any nonnegative integer $\ell$, we introduce a quotient vertex algebra $L_{\wh{\CC_\g}}(\ell,0)$ of $V_{\wh{\CC_\g}}(\ell,0)$, and  associate the integrable restricted $\wh{\mathcal{G}}(\mathcal{J})$-modules of level $\ell$
 with the $\sigma$-twisted $L_{\wh{\CC_\g}}(\ell,0)$-modules.
We also classify the irreducible $\frac{1}{2}\N$-graded $\sigma$-twisted $L_{\wh{\CC_\g}}(\ell,0)$-modules.


As usual, write
\[\{\pm 2\epsilon_1, \pm 2\epsilon_2, \pm(\epsilon_1-\epsilon_2), \pm(\epsilon_1+\epsilon_2)\}\]
for the root system of the simple Lie algebra $\mathfrak{sp}_4$. Set
\[\dot\Delta=\{\pm 2\epsilon_1, \pm 2\epsilon_2, \pm(\epsilon_1-\epsilon_2)\}.\]
For each $\beta\in \dot\Delta$, let us fix
 a root vector $e_\beta\in \mathfrak{sp}_4$ with
\[e_{2\epsilon_i}=E_{i,i+2},\ e_{-2\epsilon_i}=E_{i+2,i},\ e_{\epsilon_i-\epsilon_j}=E_{i,j}-E_{j+2,i+2}\quad (\text{$i,j=1,2$ with $i\ne j$}).\]

Let $\ell$ be a  nonnegative integer.
Write $J_{\widehat{\CC_{\g}}}(\ell,0)$ for the
$\widehat{\CC_{\g}}$-submodule of $V_{\widehat{\CC_{\g}}}(\ell,0)$ generated by the elements
\begin{equation*}
\left((e_\beta\ot t_{2}^{n})_{-1}\right)^{c_\beta\ell+1}\mathbbm{1}\quad\text{for all}\ \beta\in \dot\Delta,\ n\in \Z,
\end{equation*}
where $c_\beta=\mathrm{tr}(e_\beta e_{-\beta})\in \{1,2\}$.
We have:

\begin{lem}\label{lem:ideal-J}  $J_{\widehat{\CC_{\g}}}(\ell,0)$ is a $\sigma$-invariant graded ideal of $V_{\widehat{\CC_{\g}}}(\ell,0)$ (as a  vertex algebra).
\end{lem}
\begin{proof} Since the
generators of $J_{\widehat{\CC_{\g}}}(\ell,0)$ are homogenous and $\sigma$-invariant,
it suffices to prove  that $J_{\widehat{\CC_{\g}}}(\ell,0)$ is an ideal of $V_{\widehat{\CC_{\g}}}(\ell,0)$.
 Note that, as a
$\widehat{\CC_{\g}}$-submodule of $V_{\widehat{\CC_{\g}}}(\ell,0)$, $J_{\widehat{\CC_{\g}}}(\ell,0)$
is automatically  a left ideal of  $V_{\widehat{\CC_{\g}}}(\ell,0)$.
Furthermore, it follows from the skew symmetry of vertex algebras (\cite{LL}) that
if  $J_{\widehat{\CC_{\g}}}(\ell,0)$ is $\mathcal{D}$-invariant, then $J_{\widehat{\CC_{\g}}}(\ell,0)$ is a (right) ideal.

Let $\beta\in \dot\Delta$ and $n\in \Z$. We denote by $\mathcal{I}^n_\beta$ the subalgebra of $\widehat{\CC_\g}$
generated by the elements
\begin{equation*}
(\beta^\vee+c_\beta n\mk_{2})(m),\ (e_\beta\ot t_{2}^{n})(m),\  (e_{-\beta}\ot t_{2}^{-n})(m)\ \text{and}\ \mk_{1}(-1)
\end{equation*}
for all $m\in \Z$, where $\beta^\vee$ denotes the coroot of $\beta$.
It is easy to see that  $\mathcal{I}^{n}_\beta$   is isomorphic to the affine Lie algebra
\[\widehat{\fsl_{2}}=\fsl_2\ot \C[t,t^{-1}]\oplus \C\mk,\] with an isomorphism given by
\begin{equation*}
\begin{split}
&(e_\beta \ot t_{2}^{n})(m)\mapsto x_{+}\ot t^m , \quad (e_{-\beta}\ot t_{2}^{-n})(m)\mapsto x_{-}\ot t^m,\\
 & (\beta^\vee\ot 1 +nc_\beta \mk_{2})(m)\mapsto \alpha^{\vee}\ot t^m,\quad c_\beta\mk_{1}(-1)\mapsto \mk,\\
\end{split}
\end{equation*}

View $\mathcal{U}(\mathcal{I}_{\beta}^{n})\mathbbm{1}$  as
an $\widehat{\fsl_{2}}$-module through the above isomorphism. Note that $\mathcal{U}(\mathcal{I}_{\beta}^{n})\mathbbm{1}$
is isomorphic to the induced $\widehat{\fsl_{2}}$-module
\[V_{\widehat{\fsl_{2}}}(\ell,0)=\mathcal{U}(\widehat{\fsl_{2}})\otimes_{\mathcal{U}(\fsl_2\ot \C[t]\oplus \C\mk)}\C_\ell,
\]
where $\C_\ell$ denotes the $(\fsl_2\ot \C[t])\oplus \C\mk$-module on which $\fsl_2\ot \C[t]$ acts trivially and $\mk$ acts as the scalar $\ell$.
It is known  (\cite{LL}) that the coefficients of
\[x_\pm(z)^{\ell+1} 1\ot 1:=(\sum_{m\in \Z}x_\pm\ot t^m z^{-m-1})^{\ell+1}1\ot 1\]
are contained  in the submodule of $V_{\widehat{\fsl_{2}}}(\ell,0)$ generated by $(x_\pm\ot t^{-1})^{\ell+1}\mathbbm{1}$.
This implies that
$$((e_\beta \ot t_{2}^{n})(z))^{c_\beta\ell+1}\mathbbm{1}\in J_{\widehat{\CC_{\g}}}(\ell,0)[[z,z^{-1}]].$$
On the other hand, since
 $$\lf(e_\beta\ot t_{2}^{n}\rg)_{l}(e_\beta\ot t_{2}^{n})=0\ \text{for all}\ l\in \N,$$
we have that (\cite{LL})
\begin{align*}
&Y\lf(\left((e_\beta\ot t_{2}^{n})_{-1}\right)^{c_\beta\ell+1}\mathbbm{1},z\rg)\mathbbm{1}=Y\lf(e_\beta\ot t_{2}^{n},z\rg)^{c_\beta\ell+1}\mathbbm{1},
\end{align*}
Thus, we get that
\begin{align*}
&\mathcal{D}\lf(\left((e_\beta\ot t_{2}^{n})_{-1}\right)^{c_\beta\ell+1}\mathbbm{1}\rg)
=\Res_{z}z^{-2}Y\lf(\left((e_\beta\ot t_{2}^{n})_{-1}\right)^{c_\beta\ell+1}\mathbbm{1},z\rg)\mathbbm{1}\in J_{\widehat{\CC_{\g}}}(\ell,0).
\end{align*}
This completes the proof of lemma.
\end{proof}

In view of Lemma
\ref{lem:ideal-J}, we have an $\N$-graded vertex algebra
\begin{align}
L_{\widehat{\CC_{\g}}}(\ell,0 )=V_{\widehat{\CC_{\g}}}(\ell,0)/J_{\widehat{\CC_{\g}}}(\ell,0),
 \end{align}
equipped with an order $2$ automorphism $\sigma$.
Now we are going to associate the
 $\sigma$-twisted $L_{\widehat{\CC_{\g}}}(\ell,0)$-modules  with the integrable restricted  $\widehat{\mathcal{G}}(\mathcal{J})$-modules of level $\ell$.

\begin{definition}\label{def-int}
{\em A $\widehat{\mathcal{G}}(\mathcal{J})$-module $W$ is said to be {\em integrable} if  for $(m,n)\in S$, $x_{\pm}(m,n)$ acts locally nilpotently on $W$.}
\end{definition}
For $n\in \Z$,
set
 \begin{equation*}
 \begin{split}
 x_{\pm}(z,n)=\sum_{m\in\Z}x_{\pm}(m,n)z^{-m-1}.
 \end{split}
 \end{equation*}

\begin{prpt}\label{prop-integrable} Let $\ell$ be a complex number and let $W$ be a restricted $\widehat{\mathcal{G}}(\mathcal{J})$-module of level $\ell$. Then $W$ is integrable if and only if $\ell$ is a nonnegative integer and for all $n\in \Z$,
\begin{equation*}
 x_{\pm}(z,2n)^{\ell+1}=x_{\pm}(z,2n+1)^{2\ell+1}=0\quad \text{on  }W.
 \end{equation*}
\end{prpt}
\begin{proof}
Let $n\in\Z$. Consider  the following subalgebras of $\widehat{\mathcal{G}}(\mathcal{J})$:
\begin{eqnarray*}
\mathcal{A}_{(n)}&=&\te{Span}\{x_{+}(m,2n),x_{-}(m,-2n),\alpha^{\vee}(m,0) +4nC_2(m,0),C_1\mid m\in\Z\},\\
\mathcal{B}_{(n)}&=&\te{Span}\{x_{+}(2m,2n+1),x_{-}(2m,-2n-1),\alpha^{\vee}(2m,0) +(4n+2)C_2(2m,0),C_1\mid m\in\Z\}.
\end{eqnarray*}
One easily checks that the map
 \begin{eqnarray*}
&&x_{+}(m,2n) \mapsto x_{+}\otimes t^{m},\quad x_{-}(m,-2n) \mapsto x_{-}\otimes t^{m},\\
&&\alpha^{\vee}(m,0) +4nC_2(m,0)\mapsto \alpha^{\vee}\otimes t^{m},\quad 2C_1 \mapsto\mathbf{k}
\end{eqnarray*}
determines a Lie isomorphism from $\mathcal{A}_{(n)}$ to  $\widehat{\fsl_{2}}$,
and the map
\begin{eqnarray*}
&&x_{+}(2m,2n+1) \mapsto x_{+}\otimes t^{m},\quad x_{-}(2m,-2n-1) \mapsto x_{-}\otimes t^{m},\\
&&\alpha^{\vee}(2m,0) +(4n+2)C_2(2m,0)\mapsto \alpha^{\vee}\otimes t^{m},\quad 4C_1 \mapsto\mathbf{k}
\end{eqnarray*}
determines a Lie isomorphism from $\mathcal{B}_{(n)}$ to  $\widehat{\fsl_{2}}$.
The assertion then follows from the fact that a restricted $\widehat{\fsl_{2}}$-module  of level $\ell$  is integrable if and only if $\ell$ is a nonnegative integer and $x_{+}(z)^{\ell+1}=0=x_{-}(z)^{\ell+1}$ on it (see \cite{LL}).
\end{proof}

The following is the second main result of this paper.

\begin{thm}\label{in-re-m} Let $\ell$ be a nonnegative integer. Then the $\frac{1}{2}\Z$-graded $\sigma$-twisted $L_{\widehat{\CC_{\g}}}(\ell,0)$-modules are exactly $\Z$-graded integrable restricted $\widehat{\mathcal{G}}(\mathcal{J})$-modules of level $\ell$.
\end{thm}
\begin{proof} In view of Theorem  \ref{rem} and Proposition \ref{prop-integrable},
we only need to prove that the $\sigma$-twisted $V_{\widehat{\CC_{\g}}}(\ell,0)$-modules $W$
 on which $Y_{W}(a,z)=0$ for $a\in J_{\widehat{\mathcal{C}_{\g}}}(\ell,0)$ are exactly the restricted $\widehat{\mathcal{G}}(\mathcal{J})$-modules of level $\ell$ on which $ x_{\pm}(z,2n)^{\ell+1}=x_{\pm}(z,2n+1)^{2\ell+1}=0$
 for all $n\in \Z$.

 Let $\beta\in \dot\Delta$, $n\in \Z$.
Note that
\begin{equation*}
(e_\beta\ot t_2^n)^{(j)}_l (e_\beta\ot t_2^n)=0\quad\text{for all}\ l\in \N\ \text{and}\ j=0,1.
\end{equation*}
In view of \cite[Lemma 3.3]{CLiTW}, this implies that
\begin{equation}\label{main2:2}
Y_W(\left((e_\beta\ot t_2^n)_{-1}\right)^{c_\beta\ell+1}\mathbbm{1},z)=Y_W(e_\beta\ot t_2^n,z)^{c_\beta\ell+1}.
\end{equation}
Furthermore, a straightforward computation shows that
\begin{eqnarray*}
&&Y_W(e_{\pm 2\epsilon_1}\otimes t_2^n,z)=\frac{1}{2}z^{-\frac{1}{2}}x_{\pm}(z^{\frac{1}{2}},2n),\\
&& Y_W(e_{\pm 2\epsilon_2}\otimes t_2^n,z)=-\frac{1}{2}z^{-\frac{1}{2}}x_{\mp}(-z^{\frac{1}{2}},2n\pm 2),\\
&&Y_W(e_{\pm(\epsilon_1-\epsilon_2)}\ot t_2^n,z)=\mp\sqrt{-1}z^{-\frac{1}{2}}x_{\pm}(z^{\frac{1}{2}},2n\mp1).
\end{eqnarray*}
The theorem then follows by noting that $c_{\pm 2\epsilon_1}=1=c_{\pm 2\epsilon_2}$ and $c_{\pm(\epsilon_1-\epsilon_2)}=2$.
\end{proof}

Finally, we are aim to classify the irreducible $\frac{1}{2}\N$-graded $\sigma$-twisted $L_{\widehat{\CC_{\g}}}(\ell,0)$-modules.

\begin{prpt}\label{irr-int-hw} Let $W=\oplus_{n\in \N} W_{(n)}$ be an irreducible $\N$-graded integrable restricted $\widehat{\mathcal{G}}(\mathcal{J})$-module.
Then there is a nonzero vector $v\in W$ and a $\lambda\in \CH^*$ such that
$\widehat{\mathcal{G}}(\mathcal{J})^{+}v=0$, $\lambda(C_2)=0$, and
$hv=\lambda(h)v$ for all $h\in \CH$.
\end{prpt}
\begin{proof}
Since $W$ is irreducible, for $i=1,2$, the central element $2C_i$  acts as a scalar, say $\ell_i$, on it.
 For any $n\in \Z$,  set
 $$
\mathcal{A}(n)=\mathrm{Span}\{ x_{+}(0,n),\  x_{-}(0,-n),\ \alpha^{\vee}+2nC_2\},
 $$
which is a  subalgebra of $\widehat{\mathcal{G}}(\mathcal{J})$ that isomorphic to $\fsl_2$.
View $W$ as an $\fsl_2$-module via the isomorphism $\mathcal{A}(0)\cong \fsl_2$. The integrability of $W$ implies that
\[W=\oplus_{n\in \frac{1}{2}\Z} W(n\alpha)\qquad (W(n\alpha)=\{w\in W\mid \alpha^\vee w=2n w\}).\]
On the other hand, Proposition  \ref{prop-integrable} implies that $\ell_1\in \N$ and $x_\pm(0,0)^{\ell_1+1}=0$ on $W$.
 Thus $W(n\alpha)=0$ for all $|n|>\frac{1}{2}\ell_1$.

 For convenience, we assume that $W_{(0)}\ne 0$. Choose a nonzero vector $v\in W_{(0)}\cap W(n_0\alpha)$, where
  $n_0=\max\{n\in \frac{1}{2}\N\mid W_{(0)}\cap W(n\alpha)\ne 0\}$.
  Set $\Omega_W=\{w\in W\mid \widehat{\mathcal{G}}(\mathcal{J})^{+}w=0\}$, which is an $\CH$-submodule of $W$.
  Note that  the irreducibility of $W$ implies that $\Omega_W$ is also irreducible as an  $\CH$-module.
  And, it is easy to check that $v\in \Omega_W$.

  For every $n\in \Z$, via the isomorphism $\mathcal{A}(n)\cong \fsl_2$,
  $\mathcal{U}(\mathcal{A}(n))v$ becomes an integrable highest weight $\fsl_2$-module with the highest weight
  $(n_0+\frac{1}{2}n\ell_2)\alpha$.
This gives that $2n_0+n\ell_2\in \N$ for all $n\in \Z$, so $\ell_2=0$.
Now, as an irreducible module for the abelian Lie algebra $\CH/\C C_2$,
$\Omega_W$ must be one-dimensional.
Thus, the vector $v$ satisfies the conditions stated in proposition.
\end{proof}

Note that $L_{\widehat{\CC_{\g}}}(\ell,0 )=\C$ when $\ell=0$.
So we assume that $\ell$ is a positive integer.
 Write $P_{\ell}$ for the set consisting of those triples
\[
(\bm{\lambda},\bm{\mu},\bm{c})=((\lambda_1,\dots,\lambda_r),(\mu_1,\dots,\mu_r),(c_1,\dots,c_r))\in \N^r\times \N^r\times (\C^\times)^r,\ r=1,2,\dots,\ell,
\]
which satisfy the following conditions
\begin{itemize}
\item $\lambda_i+\mu_i>0$ for all $i=1,\dots,r$,
\item $\sum_{i=1}^s(\lambda_i+\mu_i)=\ell$, and
\item $c_1,c_2,\dots,c_r$ are distinct.
\end{itemize}
Given a  $(\bm{\lambda},\bm{\mu},\bm{c})\in P_{\ell}$  as above.
Let $\C v_{\bm{\lambda},\bm{\mu},\bm{c}}$ be a 1-dimensional $
\widehat{\mathcal{G}}(\mathcal{J})^+\oplus\CH$-module defined by
\begin{itemize}
\item $\widehat{\mathcal{G}}(\mathcal{J})^+.v_{\bm{\lambda},\bm{\mu},\bm{c}}=0$,\quad $C_2.v_{\bm{\lambda},\bm{\mu},\bm{c}}=0$,
\item $\alpha^{\vee}(0,m).v_{\bm{\lambda},\bm{\mu},\bm{c}}=(\sum_{i=1}^{r}\lambda_i c_i^m)v_{\bm{\lambda},\bm{\mu},\bm{c}}$ for all $m\in \Z$, and
\item $(2C_1(0,2m)-\alpha^{\vee}(0,2m)).v_{\bm{\lambda},\bm{\mu},\bm{c}}=(\sum_{i=1}^{r}\mu_i c_i^{m})v_{\bm{\lambda},\bm{\mu},\bm{c}}$ for all $m\in \Z$.
\end{itemize}

Form the induced $\widehat{\mathcal{G}}(\mathcal{J})$-module
\[V(\bm{\lambda},\bm{\mu},\bm{c})=\mathcal{U}(\widehat{\mathcal{G}}(\mathcal{J}))\ot_{\mathcal{U}(\widehat{\mathcal{G}}(\mathcal{J})^+\oplus \CH)}
\C v_{\bm{\lambda},\bm{\mu},\bm{c}}.\]
It is easy to see that $V(\bm{\lambda},\bm{\mu},\bm{c})$ has an irreducible quotient, which we denote as
$L(\bm{\lambda},\bm{\mu},\bm{c})$.
Let  $(\bm{\lambda}',\bm{\mu}',\bm{c}')$ be another triple in $P_{\ell}$.
 Then $L(\bm{\lambda},\bm{\mu},\bm{c})\cong L(\bm{\lambda}',\bm{\mu}',\bm{c}')$ if and only if there is a permutation $\tau\in S_r$ such that
\[(\bm{\lambda}',\bm{\mu}',\bm{c}')=((\lambda_{\tau(1)},\dots,\lambda_{\tau(r)}),
(\mu_{\tau(1)},\dots,\mu_{\tau(r)}),(c_{\tau(1)},\dots,c_{\tau(r)})).\]

Endow a canonical  $\N$-grading structure on the $\widehat{\mathcal{G}}(\mathcal{J})$-module
$V(\bm{\lambda},\bm{\mu},\bm{c})$ such that $\deg v_{\bm{\lambda},\bm{\mu},\bm{c}}=0$.
Then $L(\bm{\lambda},\bm{\mu},\bm{c})$ becomes an irreducible $\N$-graded (restricted) $\widehat{\mathcal{G}}(\mathcal{J})$-module of level $\ell$.
On the other hand, it is proved in \cite[Theorem  3.8]{CT} (see also \cite{CCT}) that
the $\widehat{\mathcal{G}}(\mathcal{J})$-module  $L(\bm{\lambda},\bm{\mu},\bm{c})$ is integrable.
Thus, by applying Theorem \ref{in-re-m},
$L(\bm{\lambda},\bm{\mu},\bm{c})$  is naturally
an irreducible $\frac{1}{2}\N$-graded $\sigma$-twisted $L_{\widehat{\mathcal{C}_{_{\g}}}}(\ell,0)$-module.
Conversely, we have the following classification result.

\begin{thm}\label{g-irr-in-re-m}Let $\ell$ be a positive integer. Then every
 irreducible $\frac{1}{2}\N$-graded $\sigma$-twisted $L_{\widehat{\mathcal{C}_{_{\g}}}}(\ell,0)$-module
 is isomorphic to  $L(\bm{\lambda},\bm{\mu},\bm{c})$ for some (unique up to a permutation)
$(\bm{\lambda},\bm{\mu},\bm{c})\in P_\ell$.
\end{thm}
\begin{proof} Let $W$ be an irreducible $\frac{1}{2}\N$-graded $\sigma$-twisted $L_{\widehat{\mathcal{C}_{_{\g}}}}(\ell,0)$-module.
By Theorem \ref{in-re-m}, $W$ is naturally  an irreducible integrable $\N$-graded $\widehat{\mathcal{G}}(\mathcal{J})$-module.
Consider the following subalgebras of $\widehat{\mathcal{G}}(\mathcal{J})$:
\begin{equation*}
\begin{split}
   & \mathfrak{a}_0=\text{Span}\{x_{+}(0,m),\ x_{-}(0,m),\ \alpha^{\vee}(0,m), \ C_2 \mid m\in \Z\},\\
    & \mathfrak{a}_1=\text{Span}\{x_{-}(1,2m),\ x_{+}(-1,2m),\ 2C_1(0,2m)-\alpha^{\vee}(0,2m), \ C_2 \mid m\in \Z\}.
\end{split}
\end{equation*}
It is obvious that both $\mathfrak{a}_0$ and $\mathfrak{a}_1$ are isomorphic to $\widehat{\fsl_2}$.

Fix a vector $v\in W$ be as in Proposition \ref{irr-int-hw}. Particulary, we have that
\[x_+(0,m).v=0=x_{-}(1,2m).v\quad \text{and}\quad h_i(m).v \in \C v\ \text{for all}\ m\in \Z, i=0,1,\]
where $h_0(m)=\alpha^{\vee}(0,m)$ and $h_1(m)=2C_1(0,2m)-\alpha^{\vee}(0,2m)$.
In view of this and the integrability of $W$, one can conclude from \cite[Lemma 4.6]{CLiTW} that for each $i=0,1$,
 there are finitely many nonzero complex numbers $b_{i,1},\dots, b_{i,k_i}$  and positive integers $p_{i,1},\dots, p_{i,k_i}$ such that
\begin{align*}
h_i(m).v=(\sum_{j=1}^{k_i}p_{i,j}b_{i,j}^m)v
\end{align*}
for all $m\in\Z$.
Let $c_1,\dots, c_r$ be all the  distinct elements in $\{b_{i,j}\mid i=0,1,
\ 1 \le j\le k_i\}$. For $k=1,\dots,r$, set $\lambda_k=p_{0,j}$ if $c_k=b_{0,j}$ for some $j=1,\dots,k_0$, and set $\lambda_k=0$ otherwise.
Similarly, set $\mu_k=p_{1,j}$ if $c_k=b_{1,j}$ for some $j=1,\dots,k_1$, and set $\mu_k=0$ otherwise.
Then one easily checks that the triple
\[(\bm{\lambda},\bm{\mu},\bm{c})=((\lambda_1,\dots,\lambda_r),(\mu_1,\dots,\mu_r),(c_1,\dots,c_r))\in P_\ell,\]
and hence $W\cong L(\bm{\lambda},\bm{\mu},\bm{c})$.
This completes the proof.
\end{proof}

\end{document}